\documentclass[12pt,reqno,oneside]{amsart}

\usepackage{amsmath}
\usepackage{amssymb}
\usepackage{amsthm}
\usepackage{graphics}
\usepackage{eucal}
\usepackage{bm}
\usepackage{bbm}
\usepackage[colorlinks]{hyperref}
\usepackage[capitalize,nameinlink,noabbrev]{cleveref}
\usepackage{mathrsfs}
\usepackage[normalem]{ulem}

\usepackage{color}
\usepackage[dvipsnames]{xcolor}

\usepackage{mathtools}
\usepackage{geometry}

\newtheorem{theorem}{Theorem}[section]
\newtheorem{definition}{Definition}[section]
\newtheorem{proposition}{Proposition}[section]
\newtheorem{lemma}{Lemma}[section]
\newtheorem{corollary}{Corollary}[section]

\newtheorem{remark}{Remark}[section]
\newtheorem{problem}{Problem}[section]

\newtheorem{dream theorem}{Dream theorem}[section]
\numberwithin{equation}{section}

\newcommand{\be}{\mathbf{e}}
\newcommand{\bu}{\mathbf{u}}

\newcommand{\SL}{\operatorname{SL}}
\newcommand{\diag}{\operatorname{diag}}

\newcommand{\supp}{\operatorname{supp}}

\newcommand{\bbZ}{\mathbb{Z}}

\newcommand{\bq}{\mathbf{q}}
\newcommand{\bx}{\mathbf{x}}
\newcommand{\bU}{\mathbf{U}}

\newcommand{\cL}{\mathcal{L}}

\newcommand{\bp}{\mathbf{p}}

\newcommand{\bv}{\mathbf{v}}
\newcommand{\bw}{\mathbf{w}}
\newcommand{\br}{\mathbf{r}}
\newcommand{\bad}{\mathbf{Bad}}

\newcommand{\bz}{\mathbf{z}}
\newcommand{\ls}[1]{{\color{RoyalBlue} \sf LS: (#1)}}

\newcommand{\ba}{\mathbf{a}}

\newcommand{\zblue}[1]{#1}

\newcommand{\by}{\mathbf{y}}

\newcommand{\Z}{\mathbb{Z}}
\newcommand{\R}{\mathbb{R}}

\newcommand{\sd}[1]{{\color{RubineRed} \sf SD: (#1)}}

\usepackage{mathtools}
\usepackage{tikz-cd}
\usepackage{graphicx}

\begin{document}
\title{Hyperplane winning for bad in manifolds}

\author{Shreyasi Datta}
\address[Shreyasi Datta]{TSMU, ISI Bangalore, India}
\email{shreyasi@isibang.ac.in, shreyasi1992datta@gmail.com}
\author{Liyang Shao}
\address[Liyang Shao]{Department of Mathematics, University of California-Berkeley, Berkeley, USA}
\email{liyang\_shao@berkeley.edu}

\date{}
\maketitle
\begin{abstract}
    We prove the hyperplane winning property of weighted badly approximable vectors inside nondegenerate analytic manifolds and affine subspaces satisfying almost optimal Diophantine conditions.  This, in particular, answers a conjecture of Badziahin--Velani for higher-dimensional manifolds, including affine cases that were understood only in $\R^2$.   
\end{abstract}

\section{Introduction}\label{sec: intro}

The goal of this paper is to explore \textit{largeness} in the sense of Schmidt for weighted badly approximable vectors inside affine subspaces and nondegenerate manifolds in $\R^n$. Weighted badly approximable vectors, as a natural higher-dimensional analogue of badly approximable numbers, are defined as follows. Given a weight vector $\bw=(w_1,\cdots,w_n)$ with positive entries and $\sum_{i=1}^nw_i=1,$ 
\[
\mathbf{Bad}(\bw):=\{\bx\in\R^n:\inf_{q\in\mathbb Z\setminus\{0\}, (p_i)\in \Z^n}|q|\max_{i=1,\cdots,n}\vert qx_i+p_i\vert^{1/w_i}>0\}.
\]

$\mathbf{Bad}(\bw)$ and its intersection with manifolds have been extensively studied over decades. Indeed, $\mathbf{Bad}(\bw)$  possesses interesting properties: it is small with respect to Lebesgue  measure (\cite{BDGW_null}), while it is large since it is of full dimension, and further, winning; see major developments in the past twenty years; see \cite{EGL, GY, LW, KW13, KW10, BFKRW, BFK, ET,  Fish09, KW2005, Da, Nesh-Simm, Lei2019, BNY21, BNY22}.

One main motivation for studying $\mathbf{Bad}(\bw)$ intersected with manifolds is Davenport's problem. In the  1960s, Davenport questioned whether there is a set of $\alpha\in\mathbb R$ of continuum size so that $(\alpha,\alpha^2)$ is badly approximable. This problem was solved in \cite{BV14} in a more general setup for $C^2$ \textit{nondegenerate} curves and affine lines in $\R^2$. In the same paper \cite{BV14}, they conjectured that the same phenomenon would appear in higher dimensions for non-degenerate curves, or even further, non-degenerate manifolds. To be exact, the following was stated in \cite{BV14} and reappeared in \cite{Beresnevich2015}, where Beresnevich answered the following problem affirmatively for non-degenerate analytic manifolds in $\R^n$, under the additional assumption $\inf_{\bw\in W}\min_i w_i>0$.

\begin{problem}[Problem 2, \cite{Beresnevich2015}]
    Given a ball $B\subset\R^d,$ a countable collection of weights $W$ in $\R^n,$ and a countable collection of functions $\mathcal F(B)$ of the form $\mathbf F:B\to\R^n,$ provide sufficient conditions (or even necessary conditions) so that 
\[
\dim\bigcap_{\mathbf F\in\mathcal F(B)}\bigcap_{\bw\in W}\mathbf F^{-1}(\mathbf{Bad}(\bw))=d.
\]

\end{problem}

The next breakthrough comes in the work of Beresnevich--Nesharim--Yang \cite{BNY22} where a stronger phenomenon, namely the  \textit{winning property} was shown for the above sets in the case of analytic curves in $\R^n.$ In particular, they answer the following problem for nondegenerate curves. 
\begin{problem}[Conjecture D,\cite{BV14}, Problem 3~\cite{BNY22}]\label{prob 2}
    Given a ball $B$ and a function $F: B \to \R^n$, find sufficient conditions on $F$ so that for any weight $\bw$, $F^{-1}(\mathbf{Bad}(\bw))$ is winning. 
\end{problem}

For affine subspaces (these are the opposite of being nondegenerate) the above problems are only understood in $\R^2$. The first kind of phenomenon that the dimension of $\mathbf{Bad}(\bw)\cap\mathcal L$ is full was observed in \cite{BPV2011}, where $\mathcal L=L_{a}$ was taken to be vertical lines in $\mathbb R^2$ with badly approximable coefficient $a.$ This was the key to solving Schmidt's conjecture in \cite{BPV2011}. It was extended to non-vertical lines in $\R^2$ in \cite{ABV}, motivated by the breakthrough by An in \cite{An}, where they show the winning of the above-mentioned set.

To the best of the authors' knowledge, in the case of $n>2$, neither of the above two problems has been settled  in the case $F$ is affine. We introduce the notion of weighted Diophantine exponent to get the optimal condition in the higher-dimensional affine case. To state our result, we need the following definition.

\begin{definition}\label{def: Diophantine}
    We say a matrix $M\in \R^{r\times s}$ is of Diophantine type $\kappa>0$ if 

\[
\inf_{\bq\in \Z^s\setminus\{\mathbf0\}} \Vert \bq\Vert^{\kappa}\min_{\bq_0\in \Z^{r}}\Vert M\bq+\bq_0\Vert>0.
\] 
 Then 

\[
\omega(M):=\inf \{\kappa:M\text{ is of Diophantine type }\kappa\},
\]
 the Diophantine exponent of $M$, which was first coined by Kleinbock in \cite{Kleinbock-extremal}.
\end{definition}

Let    \begin{equation}\label{eq: para of L}\mathcal{L}=\{(\bx, A\bx+\mathbf{b}):\bx\in\bU\subset \R^d\},\quad A\in\R^{m\times d},\ \mathbf b\in\R^m\text{ fixed},\quad m+d=n,\end{equation}
and $\bU$ is an open set in $\R^d$. Let us denote 
\[
A_{\mathbf{b}}:=\begin{bmatrix}
        \mathbf{b}\\
        A^\top 
    \end{bmatrix}
\]
 a matrix of size  $(d+1)\times m$.

 We will assume that the weight, \begin{equation}\label{eq: increasing w}
    w_1\geq \cdots\geq w_n.
\end{equation} 
We generalize the definition of $\omega(\cdot)$ in the following way. For any vector $\by\in \R^s$, and any $\nu\in \R_{>0}^s$, we write $\Vert \mathbf y\Vert_{\nu}:=\max_{i=1}^s|y_i|^{1/\nu_i}.$
\begin{definition}\label{def: weighted Diophantine}
   For $\nu\in \R^s_{>0}$ we say a matrix $M\in \R^{r\times s}$ is of $\nu$-weighted Diophantine type $\kappa$ if 

 \[
\inf_{\bq\in \Z^s\setminus\{\mathbf0\}} \Vert \bq\Vert_{\nu}^{\kappa}\min_{\bq_0\in \Z^{r}}\Vert M\bq+\bq_0\Vert>0.
\]

    Then 
\[
\omega_{\nu}(M):=\inf\{\kappa:M\text{ is of }\nu\text{-weighted Diophantine type }\kappa\},
\]
 the $\nu$-weighted Diophantine exponent of $M.$
\end{definition}

\begin{remark}\label{rem: rel between exponents}
In the case of equal weights, i.e., $\nu=(\alpha,\cdots,\alpha)\in \R^s_{>0}$ we have that 
$\omega_{\nu}(M)=\alpha \omega(M).$
\end{remark}

\begin{theorem}\label{main thm: linear}
    Let $\bw=(w_i)_{i=1}^n$ be a weight in $\R^n$ as in \eqref{eq: increasing w} and $\mathcal{L}$ be as in \eqref{eq: para of L}, with $A_{\mathbf{b}}$ is the parametrizing matrix. If 

\[
\omega_{w_{d+1},\cdots, w_{n}}(A_{\mathbf{b}})<1
\]
 then $\mathbf{Bad}(\bw)\cap\mathcal{L}$ is hyperplane absolute winning in $\mathcal{L}.$
\end{theorem}

\noindent  In case $\bw=(1/n, \cdots, 1/n)$, i.e.,  the equal-weight case, due to Remark \ref{rem: rel between exponents}, we get that the above condition in the theorem is the same as $\omega(A_{\mathbf{b}})<n$. The following remarks show that our condition is almost optimal and highly generic.

\begin{remark}\begin{itemize}
   \item We note that if $\omega_{(w_{d+1},\cdots,w_n)}(A_{\mathbf{b}})>1$, then for some $\epsilon>0$ and infinitely many $\bq_2\in\Z^m\setminus\{\mathbf0\}$, the following holds  \begin{equation}\label{eq: tight condition on A_b}\Vert A_{\mathbf{b}}\bq_2\Vert_{\Z}\leq \Vert \bq_2\Vert_{\pi_2(\bw)}^{-1-\epsilon}.\end{equation} It follows that $\bad(\bw)\cap\cL=\emptyset$ ; see \S \ref{sec:optimal}. Here, $\pi_2(\bw)=(w_{d+1},\ldots,w_n)$, and  we use the notation $\Vert\by\Vert_{\Z}:=\min_{\bp\in\Z^{d+1}}\Vert\by+\bp\Vert$ and  $|x|_\mathbb Z:=\inf_{p\in\Z}|x+p|.$
  \item 
    In particular, if $\omega(A_\mathbf{b})>\frac{1}{w_{n}}$, then $\bad(\bw)\cap \mathcal{L}$ is empty. This explains that for $n=2$ in \cite{ABV}, the above condition is almost optimal, whereas in higher dimension almost optimal condition involves a weighted Diophantine condition on $A_{\mathbf{b}}$.

    \item  In the weighted case, the corresponding generic exponent is

\[
\omega_{(w_{d+1},\ldots,w_n)}(A_{\mathbf b})=\frac{\sum_{i=d+1}^n w_i}{d+1}<1.
\]

Thus our condition holds for almost every $A_{\mathbf b}$.

\end{itemize}
\end{remark}

As an immediate corollary of Theorem \ref{main thm: linear},  using the full-dimension property of HAW sets from \mbox{
\cite[Proposition 2.3]{BFKRW}}\hskip0pt
, we obtain the following, where $\dim$ denotes Hausdorff dimension.
 \begin{corollary}
    Let $\mathcal{L}$ and $\bw$ be as in Theorem \ref{main thm: linear}, then 
 \[
\dim\bigl(\mathbf{Bad}(\bw)\cap\mathcal L\bigr)=\dim\mathcal L.
\]

\end{corollary}
Next we have an analogue to Theorem \ref{main thm: linear} for analytic nondegenerate manifolds.
\begin{theorem}\label{main thm: nondeg}
    Let $\phi:U\subset\R^d\to \R^{n}, \phi(\bx):=(\bx, \varphi(\bx))$ be an analytic and nondegenerate map, where $U$ is an open set in $\R^d$. For any weight $\bw\in \R^n$ with $w_1\geq \cdots\geq w_{n}$, $\phi^{-1}(\mathbf{Bad}(\bw))$ is hyperplane absolute winning in $U.$
\end{theorem}

The overall outline of the proof adapts the methods in \cite{BNY22} and \cite{BNY21}, which exploit linearization and a modified version of quantitative nondivergence from \cite{KM, KLW} for \textit{friendly} measures. But some natural difficulties lie in the use of quantitative nondivergence. The affine case demands a fundamentally new analysis: unlike nondegenerate manifolds, affine subspaces exhibit higher-rank obstructions that persist even for Lebesgue measure and are not addressed by existing methods; see Section \ref{sec: global affine}. Our main innovation is to show that modified nondivergence for \textit{Ahlfors regular}, \textit {absolutely decaying} measures can be applied in affine subspaces and nondegenerate manifolds of higher dimension. The method of implication for each of these cases is different. We believe results of this paper can be generalized to the inhomogeneous set-up using the ideas in \cite{DattaShao2024, DattaShao2026}. 

\section{Preliminary}\label{sec:preli}
Let $\bw=(w_i)\in\R^n$ be a weight. Without loss of generality, we may assume \begin{equation}w_1=w_2=...=w_{T_{\max}}>w_{T_{\max}+1}\geq...\geq w_n>0.\end{equation}
If $T_{\max}=n$, then we  set $w_{T_{\max}+1}=0$ . We take  $1\leq d<n$. Throughout, vector and exterior-power norms are Euclidean, and $\lambda_1(\Lambda):=\min_{\bv\in\Lambda\setminus\{0\}}\Vert\bv\Vert$. We write $K_\epsilon:=\{\Lambda\in\SL_{n+1}(\R)/\SL_{n+1}(\Z):\lambda_1(\Lambda)\geq\epsilon\}$.

For  $0<\beta<1$, $b=\beta^{-1/(1+w_1)}>1$, $y\in\R$,  $\bx\in \R^d,\br\in\R^n$ and matrix $M\in\R^{(n-d)\times d}$, while $\mathbf0\in\R^{n-d}$ is the zero vector, let us define,
\begin{equation}\label{def-of-matrix}
\begin{aligned}
    &a_y:=\textbf{diag}\{b^y,b^{-w_1y},\cdots,b^{-w_ny}\},\\&d_y:=\textbf{diag}\{\beta^{\frac{ty}{n+1}},\underbrace{\beta^{-\frac{(n+1-t)y}{n+1}},\cdots,\beta^{-\frac{(n+1-t)y}{n+1}}}_{t\text{ many}},\beta^{\frac{ty}{n+1}},\cdots,\beta^{\frac{ty}{n+1}}\},\text{where }t:=\min\{T_{\max},d\},\\&u_\mathbf{r}:=\begin{pmatrix}
        1&\br\\
        0&\mathrm{I}_n
    \end{pmatrix},  \text{ where } \mathrm{I}_n \text{ is the identity matrix of size  } n\times n,\\
    &z_M:=\begin{pmatrix}
        \mathrm I_{d+1}&\tilde{M}\\
        0&\mathrm I_{n-d}
    \end{pmatrix},\text{where }\tilde{M}=\begin{pmatrix}
        \mathbf 0\\ M^T
    \end{pmatrix}\in\mathbb R^{(d+1)\times( n-d)}, M=(M_{ij}).
    \end{aligned}
\end{equation}
Then as an  analogue  of \cite[Lemma 2.3]{BNY22}, we have the following identities.
\begin{lemma}
    \begin{equation}\label{conjugate-au}
        a_yu_{\br} a_{-y}=u_{(b^{(1+w_1)y}r_1,...,b^{(1+w_n)y}r_n)}
    \end{equation}
    \begin{equation}\label{conjugate-az}
        a_yz_Ma_{-y}=z_{(b^{(w_{i+d}-w_j)y}M_{ij})}
    \end{equation}
    \begin{equation}\label{conjugate-du}
        d_yu_{\br} d_{-y}=u_{(\beta^yr_1,...,\beta^yr_t,r_{t+1},...,r_n)}
    \end{equation}
    \begin{equation}\label{conjugate-dz}
        d_yz_M d_{-y}=z_{(\beta^{-\mathbbm 1_{\{j\leq t,i+d>t\}}y}M_{ij})}
    \end{equation}
    \begin{equation}\label{conjugate-zu}
        z_{M}u_{(\bx\text{ }\bx\cdot M^T)} z_{M}^{-1}=u_{(\bx\text{ }\mathbf0)}
    \end{equation}
    In particular, when $t=d$, \eqref{conjugate-dz} can be transformed into \begin{equation}
        d_yz_Md_{-y}=z_{\beta^{-y}M}
    \end{equation}
\end{lemma}

\begin{definition}
\label{ahlfors}
    A Borel measure $\mu$ on $\R^d$ is called $\alpha$-Ahlfors regular , for $\alpha>0$,  if there exists $A,\rho_0>0$ such that for any ball $B(\bx,\rho)\subset\R^d$ where $\bx\in\supp\mu$ and  $0<\rho\leq \rho_0$ , we have \begin{equation}\label{ahlfors property}
    A^{-1}\rho^{\alpha}\leq\mu(B(x,\rho))\leq A\rho^{\alpha}.
    \end{equation}
\end{definition}

\begin{definition}
\label{decaying}
    A Borel measure $\mu$ on $\R^d$ is called absolutely decaying if there exists $D,\delta>0$ and $r_0>0$ such that for every  $0<r\leq r_0$  and $\bx\in\supp \mu,$ every hyperplane  $H$  in $\R^d$, and $r'>0,$ we have \begin{equation}\label{decaying property}
        \mu\left(B(H,r')\cap B(\bx,r)\right)\leq D\left(\frac{r'}{r}\right)^\delta\mu(B(\bx,r)).
    \end{equation}
\end{definition}

Let us recall the definition of  hyperplane absolute winning in an open set $U\subset\R^d$  from \cite{BFKRW} and equivalent fixed-radius version of the game, as in 
\cite[Definition 7 and Proposition 8]{BNY21}. 

The game is played between two players Alice and Bob, who  take turns playing . Bob chooses a parameter $0<\beta<\frac{1}{3}$, and a ball $B_0\subset U$ of radius $r_0>0.$ Alice chooses a neighborhood $A_1$ of a hyperplane in $\R^d$ of radius  $\beta r_0$ . Then Bob chooses a ball $B_1\subset B_0\setminus A_1$ of radius $r_{1}=\beta r_0$.  At the  $N$-th step, Alice chooses a neighborhood $A_{N+1}$ of a hyperplane in $\R^d$ of radius $\beta r_{N}$ and then Bob chooses a ball $B_{N+1}\subset B_N\setminus A_{N+1}$ with radius $\beta^{N+1}r_0$. Continuing like this, we say a set $S\subset U$ is  hyperplane  absolute winning in $U$ if , for every $0<\beta<1/3$ and every initial closed ball $B_0\subset U$,  Alice has a strategy ensuring 
 \[
S\cap \bigcap_{N\geq 0} B_N\neq \emptyset.
\] 

\begin{lemma}\label{equivalence1}

   Let  $U\subset\R^d$ be open and $S\subset U$ be  a Borel set. If for  every nonzero  Ahlfors regular absolutely decaying measure $\mu$, 
\[
\supp\mu\subset U\implies S\cap\supp\mu\neq\emptyset,
\]
 then $S$ is HAW in $U.$

\end{lemma}

We will provide a proof of the above lemma in Appendix \ref{sec: Appendix}, which   follows the argument  in \cite[\S 2]{BNY21}.

In order to show the non-emptiness of $S\cap\supp\mu,$ as above we use the following notion of Cantor winning in $\supp\mu$.

\begin{definition}\label{cantor game}
    Let $\mu$ be a  nonzero  $\alpha$-Ahlfors regular measure in  $\R^d$. All balls in this game are balls in the complete metric space $X=\supp\mu$, with centres in $X$; thus $B_X(\bx,r)=X\cap B(\bx,r)$.  The $\gamma$-Cantor potential game is played by two players, say Alice and Bob, who take turns making their moves. Bob starts by choosing a parameter $0<\beta<1$, which is fixed throughout the game, and a ball $B_0\subset \supp\mu$ of radius $r_0>0$. Subsequently for $N=0,1,2,\cdots$, first, Alice chooses collections $\mathcal{A}_{N+1,i}$ of at most $\beta^{-\gamma(i+1)}$ balls of radius $\beta^{N+1+i}r_0$ for every $i\geq 0$. Then, Bob chooses a ball $B_{N+1}$ of radius $\beta^{N+1}r_0$ which is contained in $B_N$ and disjoint from $\bigcup_{0\leq l\leq N}\bigcup_{A\in\mathcal{A}_{N+1-l,l}}A.$ If there is no such ball then Alice wins by default. Otherwise, the outcome of the game is the unique point in $\cap_{N\geq 0}B_N.$

    A set $S\subset \supp\mu$ is called $\gamma$-Cantor winning if , for every $0<\beta<1$ and every initial ball,  Alice has a strategy that ensures that she either wins by default or the outcome lies in $S.$ 
    $S\subset \supp\mu$ is called Cantor winning if it is $\gamma$-Cantor winning for some $0\leq\gamma<\alpha.$
\end{definition}
The following lemma is taken from \cite[Theorems 3.4, 4.1]{BHNS}, \cite[Theorem 21]{BNY21}.
\begin{lemma}{\cite[Theorem 21]{BNY21}}\label{non-empty}
    If $S\subset \supp \mu$ is Cantor winning for some Ahlfors regular measure $\mu$ as in Definition \ref{cantor game}, then $S$ is non-empty.
\end{lemma}

The following version of non-divergence estimates is a combination of \cite[Theorem 5.11, Lemma 5.17]{BNY22}.

\begin{theorem}\label{analytic-nde}
    Let $n, d \in \mathbb{N}$, $D > 0$, $U \subset \mathbb{R}^d$ be open, $h_0 : U \rightarrow \text{SL}_{n+1}(\mathbb{R})$ be a map such that every entry of $h_0$ is a real analytic function and let $\mu$ be a locally finite Borel measure, which is absolutely decaying and Ahlfors regular on $U$. Then for any $\mathbf z \in \supp \mu \cap U$ there exists a ball $B(\mathbf z) \subset U$ centred at $\mathbf z$ and constants $C, \gamma > 0$ such that for any ball $B \subset B(\mathbf z)$ centered in $\supp\mu$, any diagonal matrix $g \in \text{SL}_{n+1}(\mathbb{R})$ and any $0 < \rho \leq 1$ at least one of the following two conclusions holds for $h(x) = gh_0(x)$
\begin{enumerate}
    \item for all $\varepsilon > 0$
    \begin{equation}
        \mu(\{x \in B : h(x)\mathbb{Z}^{n+1} \notin K_\varepsilon\}) \leq C \left(\frac{\varepsilon}{\rho}\right)^{\gamma} \mu(B).
    \end{equation}
    \item there exists a primitive collection $\mathbf{v}_1, \dots, \mathbf{v}_r \in \mathbb{Z}^{n+1}$ with $1 \leq r \leq n$ such that
    \begin{equation}
        \sup_{x \in B} \|h(x)\mathbf{v}_1 \wedge \dots \wedge h(x)\mathbf{v}_r\| < \rho.
    \end{equation}
\end{enumerate}
\end{theorem}

We finish this  section with \mbox{
\cite[Lemma 22]{BNY21}}\hskip0pt
.
 \begin{lemma}\label{covering}
    Let $\mu$ be an Ahlfors regular measure in $\R^d,$ with parameters $A, \alpha, r_0$ as in Definition \ref{ahlfors}. Let $S\subset \supp \mu$, and $0<r<r_0$ with $\mu(B(S, r))<\infty.$ Then there exists a cover of $S\cap \supp\mu$ with balls of radius $3r$ of   cardinality at most

\[
 A \mu(B(S,r))/r^{\alpha}.
\]

\end{lemma}

\section{Global estimate for affine subspaces}\label{sec: global affine}

Let $\cL$ be parametrized as in \eqref{eq: para of L} as in \S \ref{sec: intro}. We denote $\phi(\bx)=A\bx+\mathbf{b}$ and $\tilde{\phi}(\bx)=(1,\bx,\phi(\bx))=(1,\bx,A\bx+\mathbf{b}),$ where $A\in \R^{m\times d},\mathbf{b}\in\R^{1\times m}.$ Let us denote $A_{\mathbf{b}}:=\begin{bmatrix}
        \mathbf{b}\\
        A^\top 
    \end{bmatrix}$ a  matrix of size  $(d+1)\times m$.
We write $\bx$ for both column and row vector with   a slight  abuse of notation.

\begin{equation}\label{eqn: u-star(x)}
        u^\star(\bx)=\begin{bmatrix}
            1 & -\bx & \mathbf{0}\\
            0 & \mathrm{I}_d & \mathbf{0} \\
            0 & 0   & \mathrm{I}_m
        \end{bmatrix}\,.
    \end{equation}

Then 

\[
 u_1^\star(\bx)=u^\star(\bx) \begin{bmatrix}
    1 & \mathbf{0} & -\mathbf{b}\\
    0 & \mathrm{I}_{d} & A^\top\\
    0 & 0 & \mathrm{I}_{m}
\end{bmatrix}=\begin{bmatrix}
            1 & -\bx & -(A\bx+\mathbf{b})\\
            0 & \mathrm{I}_d & A^\top \\
            0 & 0   & \mathrm{I}_m
        \end{bmatrix}\,
\]

We will define 
\[
C_{A,\mathbf{b}}^*=\begin{bmatrix}
    1 & \mathbf{0} & -\mathbf{b}\\
    0 & \mathrm{I}_{d} & A^\top\\
    0 & 0 & \mathrm{I}_{m}
\end{bmatrix}.
\]

We have the following \textit{Dani-type} correspondence characterization of weighted Diophantine property. Let 
\[
\begin{matrix}
    b_{e}(T)=\diag\left(\underbrace{e^T,\ldots,e^T}_{d+1\text{ entries}},e^{-Tw_{d+1}},\ldots,e^{-Tw_n}\right),\qquad T\in\R.
\end{matrix}
\]

We have the following \textit{Dani-type} correspondence characterization of weighted Diophantine property.
\begin{lemma}\label{lem: dani0}
    For any $(w_{d+1},\cdots, w_{n})\in \R_{> 0}^{m}$ with $w_{d+1}\geq\cdots\geq w_n,$ if $A_{\mathbf{b}}$ is  of $(w_{d+1},\cdots,w_n)$ -weighted Diophantine type  $0<\kappa<1$, then, for all $T\geq0$, 
 \[
\lambda_1(b_{e}(T)C_{A,\mathbf b}^*\Z^{n+1})\gg e^{\frac{w_n(1-\kappa)}{\kappa+w_n}T}.
\]
 In particular,  taking  $\eta_{A_{\mathbf b}}=\frac{w_n(1-\kappa)}{2(\kappa+w_n)}>0$ , we have, for all sufficiently  large $T>0$,
    \begin{equation}
       \lambda_1( b_{e}(T)C^\star_{A,\mathbf{b}}\Z^{n+1})> e^{\eta_{A_{\mathbf b}} T}.
    \end{equation}
\end{lemma}

\begin{proof}
    Consider $\bv=\begin{pmatrix}
        \mathbf p\\\mathbf q
    \end{pmatrix}\in\mathbb Z^{n+1}\setminus\{\mathbf 0\},$ where $\mathbf p\in\Z^{d+1},\mathbf q=(q_{d+1},\cdots,q_n)\in\mathbb Z^{m}.$ Then  \begin{equation}\Vert b_{e}(T)C^\star_{A,\mathbf{b}}\bv\Vert\geq\max\{e^T\Vert\mathbf p+\widehat A_{\mathbf b}\mathbf q\Vert,\max_{i=d+1,\cdots,n}|e^{-Tw_i}q_i|\}.\label{eqn: lambda1-bound}\end{equation}   Here $\widehat A_{\mathbf b}=\diag(-1,1,\ldots,1)A_{\mathbf b}$ has the same weighted Diophantine type as $A_{\mathbf b}$. If $\mathbf q=0$, the left-hand side is at least $e^T$, so assume $\mathbf q\neq0$. If  $\max_{i=d+1,\cdots,n}|e^{-Tw_i}q_i|<e^{\frac{1-\kappa}{1+\frac{\kappa}{w_n}}T},$ we have 
\[
\Vert\mathbf q\Vert_{(w_{d+1},\cdots,w_n)}<e^{T+\frac{1-\kappa}{w_n+\kappa}T}
\]
by the assumption that $w_n$ is the smallest entry.

    Since $A_{\mathbf{b}}$ is $(w_{d+1},\cdots, w_n)$-weighted Diophantine type $\kappa,$ we have 
 \[
e^T\Vert\mathbf p+\widehat A_{\mathbf b}\mathbf q\Vert\gg e^T\Vert \mathbf q\Vert_{(w_{d+1},\cdots,w_n)}^{-\kappa}\geq e^{\frac{w_n(1-\kappa)}{\kappa+w_n}T}.
\] 
 Then the proof finishes by \eqref{eqn: lambda1-bound}.
\end{proof}
\begin{remark}
    It is worth pointing out that in the above lemma the conclusion is true if we replace $b_{e}(T)$ by $b_{b}(T)=\diag(b^{T},\cdots, b^T, b^{-w_{d+1}T},\cdots, b^{-w_n T}), b>1$. This is because $b_{e}(T\log b)=b_{b}(T)$, hence the above lemma applied to $T\log b$ implies $$\lambda_1(b_{b}(T)C^\star_{A,\mathbf{b}}\Z^{n+1})\gg b^{\frac{w_{n}(1-\kappa)}{\kappa+w_{n}}T}.$$ We will later use $b_{b}(T)$ abbreviated as $b(T).$

\end{remark}

Let us note that using  the definitions  of $d_{l}, a_{k}$ we get the following: 

\[
\begin{aligned}
d_la_k=\diag\Bigl(&b^k b^{-\frac{lt(1+w_1)}{n+1}},
b^{-w_1k}b^{\frac{(1+w_1)(n+1-t)l}{n+1}},\ldots,
b^{-w_tk}b^{\frac{(1+w_1)(n+1-t)l}{n+1}},\\
&b^{-w_{t+1}k}b^{-\frac{lt(1+w_1)}{n+1}},\ldots,
b^{-w_nk}b^{-\frac{lt(1+w_1)}{n+1}}\Bigr).
\end{aligned}
\]

The following proposition is the heart of this paper. The proof is quite long. We point to several subparts to make it more readable. The proof of the following proposition in the unweighted set-up can be compared with \cite[Lemma 4.4]{Shah_Yang_2024}. But in the weighted set-up, more delicate analysis of the coefficients of the higher rank wedge vectors is needed.

\begin{proposition}\label{affine-global}
   Let $\bw=(w_i)_{i=1}^n$ be a weight in $\R^n$. If $A_{\mathbf{b}}$ is  $(w_{d+1},\cdots,w_{n})$ -weighted Diophantine type $0<\kappa<1$, then for sufficiently large $s$, for all  integers $l\geq1$ and  $k\geq sl$, for every sufficiently small ball $B\subset \bU$ be a ball of  radius $r(B)<1$, such that for any primitive collection $\bv_1,\cdots, \bv_r\in \Z^{n+1}$ with  $1\leq r\leq n+1$ and $\bv:=\bv_1\wedge\cdots\wedge\bv_r\neq\mathbf0$ , we get that 
    \begin{equation}
        \sup_{\bx\in B}\Vert d_l a_{k} u_1^\star(\bx)\bv\Vert\geq \rho_{B}=r(B) 2^{-n^2}.
    \end{equation}
\end{proposition}
\begin{proof}
  The case $r=n+1$ follows since $d_la_{k}u_1^\star(\cdot)\in \mathrm{SL}_{n+1}(\R)$. We prove the remaining cases by contradiction. Suppose contrary to the conclusion of the proposition, there is a sequence with $s\to\infty$, $l\geq1$,  $k\geq sl$, a ball $B$ of radius $r(B)$ and  $\mathbf0\neq\bv=\bv_1\wedge\cdots\wedge\bv_r\in\bigwedge^r\Z^{n+1}$, with $1\leq r\leq n$ fixed after passing to a subsequence, such  that \begin{equation*}
        \sup_{\bx\in B}\Vert d_l a_{k} u_1^\star(\bx)\bv\Vert<\rho_{B}.
    \end{equation*}

The diagonal matrix 
 \[
g:=d_l a_k=\diag(b_0,b_1,\cdots, b_n)\in \mathrm{SL}_{n+1}(\R), b_1=\cdots=b_t=b^{-w_1 k}b^{\frac{(1+w_1)(n+1-t)l}{n+1}}, 1\leq t\leq d,
\] 
 $b_0=b^k b^{\frac{-lt(1+w_1)}{n+1}}$, $b_{t+1}\leq \cdots\leq b_{n}$, $b_{t+j}=b^{-w_{t+j}k}  b^{\frac{-lt(1+w_1)}{n+1}}, 1\leq j\le n-t.$  If $t<d$, take $s>(1+w_1)/(w_1-w_{t+1})$, so that $b_1\leq b_{t+1}$ within the first $d$ coordinates.

Then for $\bx=(x_1,\cdots,x_d)\in \R^d$, 
\[
g u^\star(\bx) g^{-1}=u^\star(\theta x_1,\cdots, \theta x_t, \theta_{t+1} x_{t+1},\cdots, \theta_{d} x_{d}),
\]

where 
\[
\begin{aligned}
    \theta= b_0 b_1^{-1}, \theta_{t+j}=b_0 b_{t+j}^{-1}, \quad t+j\leq d .
\end{aligned}
\]
To   simplify  notation, we denote the right-hand matrix in the above equation as 

$u^\star(\Theta \bx), \Theta=(\theta,\cdots, \theta, \theta_{t+1}, \cdots, \theta_{d})$. 
Using $u_1^\star(\bx)=u^\star(\bx) C^\star_{A,\mathbf{b}}$ and the above conjugation, we have 
\begin{equation}\label{eq: reverse}
   \sup_{\bx\in B(\mathrm{0}, r(B))}\Vert  u^\star(\Theta\bx) gC^\star_{A,b} (\bv_1\wedge\cdots\wedge\bv_r)\Vert \le \rho_{B}
\end{equation}
   Denote $V_1$ to be the span of $\be_0,\be_1,\cdots, \be_d\in \R^{n+1}$, and $V_2$ to be the span of $\be_{d+1},\cdots, \be_n\in \R^{n+1}.$

 Then we have 

\[
\bigwedge^{r}\R^{n+1}=\bigoplus_{i=0}^{d+1} \bigwedge^i V_1\bigotimes \bigwedge^{r-i}V_2.
\]

From the above, we can write 

\[
 C^\star_{A, \mathbf{b}}(\bv_1\wedge\cdots \wedge \bv_r)=\sum_{i=0}^{d+1} \bw_i,
\]
 where $\bw_i$ is a linear combination of $\sum_{J,  \# J=i}\bigwedge_{j\in J,}\be_{j}\wedge \tilde{\bw}_{J}^{(r-i)}$, and $\tilde{\bw}_{J}^{(r-i)}\in \bigwedge^{r-i} V_2.$
It is worth pointing out that $\bw_i=0$ if $r-i>m\iff i<r-m.$
Note that 
  \begin{equation}\label{eq: compare norms}1\leq \Vert \bv\Vert \ll_{A,\mathbf{b}}\Vert C^\star_{A,\mathbf{b}}\bv\Vert\asymp_n\max_{0\leq i\leq d+1}\Vert\bw_i\Vert.\end{equation}


 First we take $B$ to be centered at $\mathbf0$, just for the simplicity of the arguments. The case where $B$ is not centered at $\mathbf{0}$ is addresses at the end of the proof. 
 
\textbf{Norm of $\bw_i$:} Note that $u^\star(\bx)\be_i=\be_i-x_i\be_0, 1\leq i\leq d.$
We know $$\bw_i=\sum_{J,  \# J=i}\bigwedge_{j\in J,}\be_{j}\wedge \tilde\bw_{J}^{(r-i)},$$ where $J\subset \{0,\cdots, d\}.$
If $r-i>m$, then $\bw_i=\mathbf{0},$ hence we assume $r-i\leq m.$

Let us take $J, 0\in J$. 
\[ 
 \sup_{\bx\in B(\mathrm{0}, r(B))}\Vert \sum_{J,  \# J=i}\bigwedge_{j\in J,}gu^\star(\bx)(\be_{j}\wedge \tilde\bw_{J}^{(r-i)})\Vert \leq \rho_{B}.
\] When $0\in J$, we have two cases, either $i-1\leq t$, or $i-1>t$. If $i-1>t$, by plugging $\bx=\mathbf{0}$ in the above, we get that 
$$\begin{aligned} 
&\Vert \tilde\bw_{J}^{(r-i)}\Vert\leq b_0^{-1} b_1^{-t} (b_{t+1}^{-1}\cdots b_{i-1}^{-1})\prod_{j=1}^{r-i} b_{d+j}^{-1}\rho_{B}\\
&= b^{k(w_1+\cdots+w_{i-1}+w_{d+1}+\cdots+w_{d+r-i}-1)} b^{-l\frac{1+w_1}{1+n}t(n+1-r)}\rho_{B}
\end{aligned}.$$

If $i-1\leq t$, by plugging $\bx=\mathbf{0}$ in the above, we get that 
$$\begin{aligned} 
&\Vert \tilde\bw_{J}^{(r-i)}\Vert\leq b_0^{-1} b_1^{-(i-1)} \prod_{j=1}^{r-i} b_{d+j}^{-1}\rho_{B}\\
&= b^{k(w_1+\cdots+w_{i-1}+w_{d+1}+\cdots+w_{d+r-i}-1)} b^{-l\frac{1+w_1}{1+n}((n+1)(i-1)-tr)}\rho_{B}
\end{aligned}.$$


 We have the following expansion:

\[ 
\sum_{J,  \# J=i}\bigwedge_{j\in J,}gu^\star(\bx)(\be_{j}\wedge \tilde\bw_{J}^{(r-i)})
=\sum_{0\notin J,  \# J=i}\bigwedge_{j\in J,}g(\be_j-x_j\be_0)\wedge\tilde\bw_{J}^{(r-i)})+\sum_{0\in J} g(\bigwedge_{j\in J} \be_j\wedge \tilde\bw_{J}^{(r-i)}).\] 
 Suppose $j_0\in \{1,\cdots, d\}$, then plugging $x_{j_0}=r(B)$, and all other $x_{j}=0$, 
\[\begin{aligned}&\sum_{0\notin J, j_0\in J,  \# J=i}\bigwedge_{j\in J\setminus{j_0},}g\be_{j}\wedge g(\be_{j_0}-r(B)\be_0)\wedge \tilde\bw_{J}^{(r-i)}+ \sum_{0, j_0\notin J,  \# J=i}\bigwedge_{j\in J}g\be_{j} \wedge \tilde\bw_{J}^{(r-i)}\\ & ~~~~~~~~~~~~~~~~~~~~~~~~~~~+\sum_{0\in J} g(\bigwedge_{j\in J} \be_j\wedge \tilde\bw_{J}^{(r-i)})\end{aligned}.\]
Now using orthogonality, we conclude from \eqref{eq: reverse} that for any $J$ with $0\notin J$ and containing $j_0$,
\[\left\Vert -r(B) g\left(\wedge_{j\in J\setminus \{j_0\}}\be_j\wedge \be_0\wedge\tilde\bw_{J}^{(r-i)}\right)+g\left(\wedge_{j\in J\setminus\{j_0\}}\be_j\wedge\be_0\wedge \wedge \tilde\bw_{J\setminus\{j_0\}\cup\{0\}}^{(r-i)}\right)\right\Vert \leq \rho_{B}\]

using triangle inequality, we get that for any such $J$
\[\begin{aligned}\Vert \tilde\bw_{J}^{(r-i)}\Vert 
\leq 2 \rho_{B}r(B)^{-1} b^{k(\sum_{j\in J\setminus\{j_0\}} w_{j}+w_{d+1}+\cdots +w_{d+r-i}-1)} b^{l\eta_1},\end{aligned}\] for some $\eta_1=\frac{1+w_1}{1+n}\max\{-t(n+1-r), -((n+1)(i-1)-tr)\} \in \R$.
From the above discussion and bounds of $\Vert\tilde\bw_{J}^{(r-i)}\Vert $, we may consider $s$ to be large enough (independent on $B$) such that 
\begin{equation}\label{eq: norm of w_i}\Vert \bw_{i}\Vert\ll_{A,\mathbf{b}} 1,\end{equation} where the implied constant appears in \eqref{eq: compare norms}.


\textbf{Norm of $\bw_1$:} For the norm of $\bw_1$, we want to show a sharper bound, which is crucial in the weighted case, but can be avoided in the unweighted case. 
We want to expand $\bw_1$ further, as follows:
\[
\bw_1=\sum_{i=0}^{d}\be_i\wedge\widetilde \bw_i^{(r-1)}
    =\sum_J\sum_{i=0}^{d}W_{iJ}\,\be_i\wedge \be_J,
\]
where \(J\subset\{d+1,\ldots,n\}\) and \(\#J=r-1\). Thus
\[
g\bw_1=\sum_J\sum_i W_{iJ}\,g(\be_i\wedge \be_J).
\]
By plugging $\bx=\mathbf{0}$ in \eqref{eq: reverse}, we get that 
\begin{equation}\label{eq: bound of W_{0J}}
|W_{0J}|\le b_0^{-1}\prod_{j\in J}b_j^{-1} \rho_{B}.
\end{equation} for all $J\subset \{d+1,\cdots,n\}, \#J=r-1.$

For \(i_{0}\ne0, 1\le i_0\le d\), let us choose \(x_{i_0}=r(B)\) and \(x_i=0\) for
\(i\ne i_0\).
Using the definition of matrices, we have
\begin{align*}
&\sum_J\sum_i g u^\star(\bx)(W_{iJ}\be_i\wedge \be_J)\\
&=\sum_J\sum_{i\ne0}
  g\bigl(W_{iJ}(\be_i-x_i \be_0)\wedge \be_J\bigr)
  +\sum_J g(W_{0J}\be_0\wedge \be_J)\\
&=\sum_J
  g\bigl(W_{i_0J}(\be_{i_0}-r(B)\be_0)\wedge \be_J\bigr)
  +\sum_J\sum_{\substack{i\ne i_0\\ i\ne0}}
  g(W_{iJ}\be_i\wedge \be_J)
  +\sum_J g(W_{0J}\be_0\wedge \be_J)\\
&=\sum_J\sum_{i\ne0}g(W_{iJ}\be_i\wedge \be_J)
  +\sum_J g\bigl((W_{0J}-r(B)W_{i_0J})
                    \be_0\wedge \be_J\bigr).
\end{align*}

By \eqref{eq: reverse} and orthogonality,
\[
\left\|
g\bigl((W_{0J}-r(B)W_{i_0J})\be_0\wedge \be_J\bigr)
\right\|\le \rho_{B}.
\]
Hence by \eqref{eq: bound of W_{i,J}} \begin{equation}\label{eq: bound of W_{i,J}}
     \vert W_{i_0J}\vert\leq 2\rho_{B} r(B)^{-1} b_0^{-1}\prod_{j\in J} b_{j}^{-1}< b_0^{-1}\prod_{j\in J} b_{j}^{-1},
\end{equation}
for every $i_0, J$ as above.

\textbf{Norm of $\bw_0$:} By plugging $\bx=\mathbf{0}$ in \eqref{eq: reverse} and since $\bw_i$'s, $0\leq i\leq d+1$ belong to mutually orthogonal spaces, we get that \[\Vert g\bw_0\Vert\leq \rho_{B}<1\implies \Vert \bw_0\Vert \leq b^{-1}_{d+1}\cdots b^{-1}_{d+r}.\] Plugging values of $b_i$, we get \begin{equation}
    \label{ineq-on-w0}
    \Vert \bw_0\Vert<b^{k\sum_{i=d+1}^{r+d}w_i}\cdot b^{trl\frac{1+w_1}{1+n}}.
\end{equation} 
Moreover, we may write \begin{equation}\label{eq: expression of w_0}\bw_0=\sum_{J', J'\subset \{d+1,\cdots,n\}, \#J'=r} W^{0}_{J'}\be_{J'}.\end{equation} Then by plugging $\bx=\mathbf{0}$, we get that 
\begin{equation}\label{eq: bound on w^0_J}
    \vert W^{0}_{J'}\vert \leq \prod_{j\in J'}b_{j}^{-1} \rho_{B}<\prod_{j\in J'}b_{j}^{-1}.
\end{equation}

We claim that $\Vert \bw_0\Vert \neq 0.$ This follows by \eqref{eq: compare norms} and \eqref{eq: norm of w_i}, 
$$
1\ll\Vert C^\star_{A,\mathbf{b}}\bv\Vert\asymp_{n}\Vert \bw_0\Vert.$$ 

By the action of $C_{A,\mathbf{b}}^\star$ on $V_2$, as $\bv\in \bigwedge^r\Z^{n+1}$, we get that 
$\bw_0\in \bigwedge^r\Z^{n+1}.$

Thus \begin{equation}\label{w0-estimate}
    r\leq m,
\end{equation}where the second inequality is from that $r>m$ immediately implies $\bw_0=0$ by definition of $\bw_0.$

\textbf{Contradiction to Diophantine condition:}
Recall $$b(T):=\diag(b^T, \cdots, b^T, b^{-Tw_{d+1}},\cdots, b^{-Tw_{n}}) $$ and  let us denote $\beta^\star=\frac{(1+w_1)t}{n+1}.$ 
For any $T\leq k,$ 

\begin{equation}\label{btw0-estimate}\Vert b(T)\bw_0\Vert \asymp_{n}\max_{J', \#J'=r, J'\subset\{d+1,\cdots,n\}} \Vert W_{J'}^{0} b(T)  \be_{J'}\Vert \leq b^{-(T-k)\sum_{i=d+1}^{d+r} w_i} b^{\beta^\star rl},\end{equation} using \eqref{eq: bound on w^0_J}, \eqref{eq: increasing w} and by expression of $\bw_0$ as in \eqref{eq: expression of w_0}. 

Similarly, we have for $T\leq k$ \begin{equation}\label{eq: bound on b(T)w_1}\Vert b(T)\bw_1\Vert\leq c_n^{0} b^{(T-k)(1-\sum_{i={d+1}}^{d+r-1}w_i)} b^{\beta^*rl}=:W_1(T)  \end{equation} by \eqref{eq: bound of W_{i,J}} and the expression of $\bw_1$ in terms of basis $\be_i\wedge\be_{J}$, where $c_n^{0}$ is a positive constant dependent on $n$.  



Note that for any $T\ge 0$, we have that $\Vert b(T)\bw_0\Vert\geq b^{-T} $ since $ \Vert\bw_0\Vert\geq 1 $ and $(w_i)$ is a weight. From \eqref{eq: bound on b(T)w_1} and \eqref{btw0-estimate}, we get that 
for $$ T\ll_{n}
\frac{
\left(1-\sum_{i=d+1}^{d+r-1}w_i\right)k
-\beta^\star rl
}{
2-\sum_{i=d+1}^{d+r-1}w_i
}=:k^\star, \quad  \Vert b(T)\bw_1\Vert \leq W_1(T)\leq \Vert b(T)\bw_0\Vert.$$ We take $s$ large enough so that $k^\star>\frac{
\left(1-\sum_{i=d+1}^{d+r-1}w_i\right)k
}{2(
2-\sum_{i=d+1}^{d+r-1}w_i
)} c^{1}_{n}$, where $c^{1}_{n}$ is $n$ dependent constant in the previous display. This in particular implies there exists $\left(0, k^\star \right) \subset (0,k]$ with the above property. This is because the map $T\to \Vert b(T)\bw_0\Vert$ is decreasing and the map $T\to W_1(T)$ is increasing. Also note that $W_1(k)=c_{n}^{0}b^{\beta^\star r l}$, and $\Vert b(k)\bw_1\Vert\leq b^{\beta^\star rl}.$ Thus there exists a $T^\star \leq k$ such that $\Vert b(T^\star)\bw_0\Vert=W_1(T^\star)$.  Note that by the choice of $T^\star$, $T^\star\geq k^\star\implies T^\star\gg k.$

From this choice of $T^\star$, 
$$\Vert b(T^\star)\bw_1\Vert\ll_n \Vert b(T^\star)\bw_0\Vert=W_1(T^\star) \ll_{n} b^{(T^\star-k)(1-\sum_{i={d+1}}^{d+r-1}w_i)} b^{\beta^*rl}<b^{\beta^\star r l}. $$  
Thus we have 
\begin{equation}\label{eq:bound_b(T)w_0_special}
    \Vert b(T^\star)\bw_0\Vert^{1/r} \ll_{n} b^{\beta^\star l}.
\end{equation}


Furthermore, using \cite[Lemma 4.4]{Shah_Yang_2024}, we get that 
$$\Vert b(T^\star)\bw_i\Vert \Vert b(T^\star)\bw_0\Vert\ll_{n} \Vert b(T^\star)\bw_{i-1}\Vert\Vert b(T^\star)\bw_1\Vert, i\geq 2.$$
Thus, by the choice of $T^\star$, and above 
$$\Vert b(T^\star)\bw_i\Vert\ll_{n} \Vert b(T^\star)\bw_{i-1}\Vert, \quad i\geq 1.$$

Therefore, $$\Vert b(T^\star)C^\star_{A,\mathbf{b}}(\bv_1\wedge\cdots\wedge\bv_r)\Vert \ll_{n}\Vert b(T^\star)\bw_0\Vert.$$
Hence using Minkowski's theorem there exists a $\tilde\bv\in \Z^{n+1}$ such that 
\begin{equation}\label{eq: short vector}
\Vert b(T^\star)C^\star_{A,\mathbf{b}}\tilde\bv\Vert\ll_{n} \Vert b(T^\star) 
\bw_0\Vert^{1/r}    
\stackrel{\eqref{eq:bound_b(T)w_0_special}}{\ll_{n}}
b^{\beta^\star l}.\end{equation}
Since $T^\star\gg_{n, (w_i)_{i=1}^n} k$, for $\eta=\frac{1}{2}\eta_{A_{\mathbf b}}>0$, by taking large $s$, we have $$\Vert b(T^\star)C^\star_{A,\mathbf{b}}\tilde\bv\Vert <b^{\eta T^\star},$$ where $\eta_{A_{\mathbf b}}$ appears in Lemma \ref{lem: dani0}. This contradicts the Diophantine property of $A_{\mathbf{b}}$ by Lemma \ref{lem: dani0}.

\textbf{General ball:} If $B$ is centered at $\bx_0\in B$, then we use the following reformulation of $d_{l}a_{k}u_1^\star(\bx)$ to be

\[
u_1^\star(\bx)=u^\star(\bx) C^\star_{A,\mathbf{b}}=u^\star(\bx-\bx_0) u^\star(\bx_0) C^\star_{A,\mathbf{b}}.
\]

Thus \eqref{eq: reverse} will be replaced by 
\begin{equation}\label{eq: reverse translated}
\sup_{\by\in B(\mathbf{0}, r(B))}\Vert u^\star(\Theta \by) g u^\star(\bx_0) C^\star_{A,\mathbf{b}}(\bv_1\wedge\cdots\wedge\bv_r) \Vert \leq \rho_{B}.
\end{equation}
Using the same decomposition as above, we now write
\[
u^\star(\bx_0)C^\star_{A,\mathbf{b}}
(\bv_1\wedge\cdots\wedge\bv_r)
=\sum_{i=0}^{d+1}\bw_i,
\]
where
\(
\bw_i\in \bigwedge^iV_1\otimes\bigwedge^{r-i}V_2
\).
Since $u^\star(\bx_0)$ preserves $V_1$ and acts trivially on $V_2$, $\bw_0\in\bigwedge^r\Z^{n+1}$. Moreover,
\[
1\leq\Vert\bv\Vert
\ll_{A,\mathbf b,\bU}
\Vert u^\star(\bx_0)C^\star_{A,\mathbf b}\bv\Vert.
\]
Therefore we repeat the preceding estimates, using \eqref{eq: reverse translated} and choosing $\by=\mathbf 0$ or $\by=r(B)\be_j$ in place of the corresponding choices of $\bx$. Thus, for every $\eta>0$, there exist $T\to\infty$ and a nonzero $\bz\in\Z^{n+1}$ such that
\[
\Vert b(T)u^\star(\bx_0)C^\star_{A,\mathbf b}\bz\Vert
\leq e^{\eta T}.
\]
Finally, since $b(T)$ commutes with $u^\star(\bx_0)$, we get
\[
\Vert b(T)C^\star_{A,\mathbf b}\bz\Vert
\ll_{\bU} e^{\eta T},
\]
which contradicts Lemma \ref{lem: dani0} after taking $\eta>0$ sufficiently small. This completes the proof for an arbitrary ball $B$.

\end{proof}

\subsection{Almost optimal condition} \label{sec:optimal} 
Suppose $A_{\mathbf{b}}$ satisfies \eqref{eq: tight condition on A_b} for infinitely many 
nonzero $\bq_2\in\Z^m$, with a fixed $\epsilon>0$.
If $(\bx,A\bx+\mathbf b)\in\mathbf{Bad}(\bw)$, weighted transference gives a constant $c(\bx)>0$ such that

\[
|\bq_1\cdot\bx+\bq_2\cdot(A\bx+\mathbf b)+q_0|\geq\frac{c(\bx)}{\Vert(\bq_1,\bq_2)\Vert_{\bw}}
\]

for every $(\bq_1,\bq_2)\neq0$ and $q_0\in\Z$. Choose $q_0,\bq_1$ minimizing $\Vert A_{\mathbf b}\bq_2+(q_0,\bq_1)^\top\Vert$. Along the approximating sequence this norm is at most $1$, so $\Vert\bq_1\Vert\ll_A\Vert\bq_2\Vert$. By \eqref{eq: increasing w},

\[
\Vert\bq_1\Vert_{\pi_1(\bw)}\ll_A\Vert\bq_2\Vert^{1/w_d}\ll_{A,n}\Vert\bq_2\Vert_{\pi_2(\bw)},
\]

where $\pi_1(\bw)=(w_1,\ldots,w_d)$ and $\pi_2(\bw)=(w_{d+1},\ldots,w_n)$. Therefore

\[
\Vert A_{\mathbf b}\bq_2\Vert_{\Z}\gg_{\bx,A}\frac{1}{\Vert\bq_2\Vert_{\pi_2(\bw)}}.
\]

This contradicts \eqref{eq: tight condition on A_b}  as $\Vert\bq_2\Vert\to\infty$. Hence $\bad(\bw)\cap\mathcal L=\emptyset$ . In particular, if $\omega(A_{\mathbf b})>1/w_n$, choose $1/w_n<\kappa<\omega(A_{\mathbf b})$ and then $\epsilon>0$ with $(1+\epsilon)/w_n<\kappa$. Infinitely many nonzero $\bq_2$ satisfy $\Vert A_{\mathbf b}\bq_2\Vert_{\Z}\leq\Vert\bq_2\Vert^{-\kappa}$, which implies \eqref{eq: tight condition on A_b}, after decreasing $\epsilon$ if necessary.

\section{Global estimate for nondegenerate manifolds}\label{sec: global nondeg}
We prove the global estimate for  the nondegenerate case in this section. The proof of the following is motivated by a more complicated situation that arises in \cite{BD_upcoming}.

\begin{proposition}\label{nondeg-global}
   Let $\phi(\bx)=(\bx,\varphi(\bx)):\bU\subset\R^d\to\R^n$ be analytic and nondegenerate at $\bx_0\in\bU$, where $\bU$ is open,
   then for every sufficiently small closed ball $B_0$ centered at $\bx_0$ and contained in $\bU$, there exists $\rho=\rho(\phi,B_0)>0$, such that for any $g=\diag(\lambda_0,\ldots,\lambda_n)\in\SL_{n+1}(\R)$ with 
    \begin{equation}\label{eq: cond on lambda}
       \lambda_0\geq 1, \quad 0<\lambda_i\leq 1, \quad 1\leq i\leq n,
   \end{equation}   and for any primitive collection  $\bv_1,\ldots,\bv_r\in\Z^{n+1}$ with $1\leq r\leq n+1$,
    \begin{equation}
       \sup_{\bx\in B_0}\Vert gz_{D\varphi(\bx)}u_{(\bx\text{ }\varphi(\bx))}(\bv_1\wedge\cdots\wedge\bv_r)\Vert \geq \rho.
   \end{equation}

\end{proposition}
  \begin{proof}
    Let $H(\bx)=gz_{D\varphi(\bx)}u_{(\bx,\varphi(\bx))}$ and $\tilde{\phi}(\bx)=(1,\bx,\varphi(\bx))$. The case $r=n+1$ follows from $\det H(\bx)=1$. For $1\leq r\leq n$, let $\bv=\bv_1\wedge\cdots\wedge\bv_r$ be primitive. By \mbox{
\cite[Lemmas 5.20--5.21]{BNY22}}\hskip0pt
, there  is a primitive collection  $\bu_1,\ldots,\bu_{n+1-r}\in\Z^{n+1}$ whose wedge $\bu=\bigwedge_{i=1}^{n+1-r} \bu_i$ satisfies \cite[Equation 5.7, Equation 5.8]{BNY22},  and is the Hodge dual of $\bv$ upto sign.

   Nondegeneracy implies that the coordinate functions of $\tilde{\phi}$ are linearly independent on $B_0$. By compactness of the unit sphere, there is $\rho_1>0$, independent of $\ba$, such that
    \begin{equation}\label{uniform-linear-bound}
        \sup_{\bx\in B_0}|\tilde{\phi}(\bx)\cdot\ba|\geq\rho_1\Vert\ba\Vert,\qquad \ba\in\R^{n+1}.
    \end{equation}

    \noindent Let $s$ be the least nonnegative integer for which there are distinct indices $j_1,\ldots,j_{r-1-s}\in\{d+1,\ldots,n\}$ such that
    \begin{equation}\label{eq: nonzero wedge}
        \bw:=\be_{j_1}\wedge\cdots\wedge\be_{j_{r-1-s}}\wedge\bu\neq0.
    \end{equation} 

 Such an $s$  exists, since $s=r-1$ is allowed. Suppose $s=0$  , which  implies $r-1\leq m$.  Then $\bw$ is a wedge of $n$ vectors with integer coordinates. Its Hodge dual is a nonzero integer vector $\ba$, so

  \[
|\tilde{\phi}(\bx)\cdot\ba|=\Vert\tilde{\phi}(\bx)\wedge\bw\Vert.
\] 

    For $I=\{0,j_1,\ldots,j_{r-1}\}$, the row formula gives
     
 \[
|(H(\bx)\bv)_I|=\lambda_0\prod_{k=1}^{r-1}\lambda_{j_k}\Vert \tilde\phi(\bx)\wedge\bw\Vert=\lambda_0\prod_{k=1}^{r-1}\lambda_{j_k}|\tilde{\phi}(\bx)\cdot\ba|.
\] 

    Since  $g\in \SL_{n+1}(\R)$, by \eqref{eq: cond on lambda} we have that $\lambda_0 \prod_{k=1}^{r-1}\lambda_{j_k}\geq 1$.
    Thus \eqref{uniform-linear-bound} and $\Vert\ba\Vert\geq1$ give

\[
\sup_{\bx\in B_0}\Vert H(\bx)\bv\Vert\geq\rho_1.
\]  Now 
 suppose  $s>0$. Minimality of $s$ implies that $\be_j\wedge\bw=0$ for every $d+1\leq j\leq n$. Thus the subspace represented by $\bw$, which has dimension $n-s$, contains the entire span of $\be_{d+1},\ldots,\be_n$. In particular $s\leq d$. Let us choose distinct $i_1,\ldots,i_s\in\{1,\ldots,d\}$  such that
 \[
\be_{i_1}\wedge\cdots\wedge\be_{i_s}\wedge\bw\neq0.
\] 

   Note that by the parametrization $\phi(\bx)=(\bx, \varphi(\bx)),$
 Since $\partial_i\tilde{\phi}(\bx)-\be_i$ belongs to the subspace generated by $\bw$, hence
    \[
    \tilde{\phi}(\bx)\wedge\partial_{i_1}\tilde{\phi}(\bx)\wedge\cdots\wedge\partial_{i_s}\tilde{\phi}(\bx)\wedge\bw
    =\tilde{\phi}(\bx)\wedge\be_{i_1}\wedge\cdots\wedge\be_{i_s}\wedge\bw.
    \]
    Let $\ba\in\Z^{n+1}\setminus\{0\}$ be the Hodge dual of $\be_{i_1}\wedge\cdots\wedge\be_{i_s}\wedge\bw$. With

  \[
I=\{0,i_1,\ldots,i_s,j_1,\ldots,j_{r-1-s}\},
\]

    the same row formula now yields

  \[
|(H(\bx)\bv)_I|=\lambda_0\prod_{k=1}^s\lambda_{i_k}\prod_{k=1}^{r-1-s}\lambda_{j_k}|\tilde{\phi}(\bx)\cdot\ba|.
\] 
Applying the linear independence criteria uniformly over the ball $B_0$ as in \eqref{uniform-linear-bound} and by \eqref{eq: cond on lambda} we obtain:
    \[
        \\sup_{\bx\in B_0} \|H(\bx)(\bv_1 \wedge \dots \wedge \bv_r)\| \geq  \rho_1 > 0.
    \]

     Taking $\rho=\min\{1,\rho_1\}$ completes the proof.
 \end{proof}

\section{Reduction to Cantor winning}

The  proofs  of Theorem \ref{main thm: linear} and Theorem \ref{main thm: nondeg}  follow from \mbox{
Lemma \ref{equivalence1} }\hskip0pt
once we show that $\phi^{-1}(\mathbf{Bad}(\bw))\cap\supp\mu\neq\emptyset$ for each nonzero Ahlfors regular absolutely decaying measure under consideration.

From now on, we fix a nonzero Ahlfors regular absolutely decaying measure $\mu$ with $X=\supp\mu\subset \bU$.

The nonemptiness of $\phi^{-1}(\textbf{Bad}(\bw))\cap\supp\mu$ would follow from Lemma \ref{non-empty} if we can show it is Cantor winning. Therefore, the rest of this paper will be devoted to providing a winning strategy for Cantor potential game over $\phi^{-1}(\textbf{Bad}(\bw))\cap\supp\mu$.

  We first state a few simplifying assumptions for the game strategy. When the potential game starts, the player Bob would choose a ball $B_0$ in $X$ of radius $r_0$ with center at $\bx_0\in X$.  By covering arguments in \cite{BNY22}, we may assume $$B(\bx_0,5r_0)\subset B(\mathbf z),$$ where $B(\mathbf z)$ is a ball that satisfies Theorem \ref{analytic-nde}. Further, $r_0$ is small enough so that $$r_0\leq \min\left\{\frac{1}{3}\rho_0,\frac{1}{\sqrt n},A^{-1/\alpha}\right\},$$ where $\rho_0,A,\alpha$ are as in Definition \ref{ahlfors}; in particular, for any $\bx\in\supp\mu,r\leq r_0,$ one has $\mu(B(\bx,r))\leq 1.$ On the other hand, we may assume that $\beta>0$ is sufficiently small(which will be determined later) via postponing Alice's strategy. The details of these simplifying assumptions can be found in \cite[after Proposition 26]{BNY22}.

\section{Measure estimates using global and local}

In this section, we aim to provide measure estimates needed for the potential game strategy. Measure estimates in this section hold for both nondegenerate manifolds and affine subspaces.

We consider the manifold $\phi(\bx)=(\bx\text{ }\varphi(\bx)),$ where $\varphi(\bx)$ is as in Theorem \ref{main thm: nondeg} or respectively  $\varphi(\bx)=A\bx+\mathbf b$ as in Theorem \ref{main thm: linear}.

As an analogue  of \cite[Lemma 3.1]{BNY22}, we have the following lemma .
 \begin{lemma}\label{transition}
     For every sufficiently small $\beta>0$, any $\tilde{\bx}_N,\bx\in B(\bx_0,5r_0)$, and nonnegative integers $l,k,N$  such that  $k+1\leq 2N,l+1\leq N$ and $\Vert \bx'\Vert=\lVert \bx-\tilde{\bx}_N\rVert\leq r=r_0\beta^N$, and for any  nonzero $\bv=\bv_1\wedge\cdots\wedge\bv_i\in\bigwedge^i\R^{n+1}$, $1\leq i\leq n+1$,  we have \begin{equation}
        \frac{1}{2}\leq\frac{\Vert u_{(\beta^{l-k}x_1',...,\beta^{l-k}x_t',b^{(1+w_{t+1})k}x_{t+1}',...,b^{(1+w_d)k}x_d',0,...,0)}H(\tilde{\bx}_N)\bv\Vert}{\Vert H(\bx)\bv\Vert}\leq 2
    \end{equation}
    where  \begin{equation}
        H(\bx):=d_la_kz_{D\varphi(\bx)}u_{(\bx\text{ }\varphi(\bx))}.
    \end{equation} 
 The same comparison holds if instead $l\geq1$, $k\geq3l$ and $\Vert\bx-\tilde{\bx}_N\Vert\leq r=r_0\beta^{k-l}$.
 \end{lemma}
\begin{proof}
 Note that by Taylor series, we have the following estimate
    \begin{equation}\label{taylor-u}\begin{split}
        u_{(\bx\text{ }\varphi(\bx))}=&u_{(\bx\text{ }\varphi(\tilde\bx_N)+\bx'\cdot D\varphi(\tilde\bx_N)^T+O(r^2))}\\=&u_{(\bx'\text{ }\bx'\cdot D\varphi(\tilde\bx_N)^T+O(r^2))}u_{(\tilde\bx_N\text{ }\varphi(\tilde\bx_N))}.
    \end{split}
    \end{equation}
    Similarly
    \begin{equation}\label{taylor-z}
        \begin{split}
            z_{D\varphi(\bx)}&=z_{D\varphi(\tilde\bx_N)+O(r)}\\&=z_{D\varphi(\tilde\bx_N)}z_{O(r)}.
        \end{split}
    \end{equation}
    Then
    \begin{equation}\label{expand-1}
        \begin{split}
            d_la_kz_{D\varphi(\bx)}\stackrel{\eqref{taylor-z}}{=}&d_la_kz_{O(r)}z_{D\varphi(\tilde\bx_N)}\\\stackrel{\eqref{conjugate-az}}{=}&d_lz_{O(r)}a_kz_{D\varphi(\tilde\bx_N)}\\\stackrel{\eqref{conjugate-dz}}{=}&z_{O(r\beta^{-l})}d_la_kz_{D\varphi(\tilde\bx_N)}.
        \end{split}
    \end{equation}
    Further
    \begin{equation}
        \label{expand-2}
        \begin{split}
           & d_la_kz_{D\varphi(\tilde\bx_N)}u_{\bx\text{ }\varphi(\bx)}\\ \stackrel{\eqref{taylor-u}}{=}& d_la_kz_{D\varphi(\tilde\bx_N)}u_{(\bx'\text{ }\bx'\cdot D\varphi(\tilde\bx_N)^T+O(r^2))}u_{(\tilde\bx_N\text{ }\varphi(\tilde\bx_N))}\\\stackrel{\eqref{conjugate-zu}}{=}&d_la_ku_{O(r^2)}u_{(\bx'\text{ }\mathbf{0})}z_{D\varphi(\tilde\bx_N)}u_{(\tilde\bx_N\text{ }\varphi(\tilde\bx_N))}
            \\\stackrel{\eqref{conjugate-au}}{=}&d_lu_{O(r^2\beta^{-k})}u_{(b^{(1+w_1)k}x_1',...,b^{(1+w_d)k}x_N',0,...,0)}a_kz_{D\varphi(\tilde\bx_N)}u_{(\tilde\bx_N\text{ }\varphi(\tilde\bx_N))}\\\stackrel{\eqref{conjugate-du}}{=}&u_{O(r^2\beta^{-k})}u_{(\beta^{l-k}x_1',...,\beta^{l-k}x_t',b^{(1+w_{t+1})k}x_{t+1}',...,b^{(1+w_d)k}x_d',0,...,0)}d_la_kz_{D\varphi(\tilde\bx_N)}u_{(\tilde\bx_N\text{ }\varphi(\tilde\bx_N))}.
        \end{split}
    \end{equation}
    On combining \eqref{expand-1}, we conclude that
    \begin{equation}
        \begin{split}
            &d_la_kz_{D\varphi(\bx)}u_{(\bx\text{ }\varphi(\bx))}\stackrel{\eqref{expand-1}}{=}z_{O(r\beta^{-l})}d_la_kz_{D\varphi(\tilde\bx_N)}u_{\bx\text{ }\varphi(\bx)}\\\stackrel{\eqref{expand-2}}{=}&z_{O(r\beta^{-l})}u_{O(r^2\beta^{-k})}u_{(\beta^{l-k}x_1',...,\beta^{l-k}x_t',b^{(1+w_{t+1})k}x_{t+1}',...,b^{(1+w_d)k}x_d',0,...,0)}d_la_kz_{D\varphi(\tilde\bx_N)}u_{(\tilde\bx_N\text{ }\varphi(\tilde\bx_N))}.
        \end{split}
    \end{equation}
     Now $z_{O(r\beta^{-l})}u_{O(r^2\beta^{-k})}$ is arbitrarily close to the identity matrix for both cases where either $N>1,r=r_0\beta^{N},k<2N,l<N$ or $r=r_0\beta^{k-l},k\geq 3l$ when $\beta$ is small enough, which finishes the proof.
 \end{proof}

Denote  \begin{equation}
    H_0(\bx):=z_{D\varphi(\bx)}u_{(\bx\text{ }\varphi(\bx))},\quad\bx\in U,
\end{equation} 
\begin{equation}
    A_{\epsilon}^{k,l}(B):=\{\bx\in B:d_la_kH_0(\bx)\Z^{n+1}\notin K_{\epsilon}\},
\end{equation}
and in particular, we write $A_{\epsilon}^{k,l,N}$ for $A_{\epsilon}^{k,l}(B_N)$.  Let $\be_+=\be_0$. For any exterior vector $\bv$, write $\bv=\bv_++\bv_-$ for its orthogonal decomposition into terms containing $\be_0$ and terms not containing $\be_0$.

The following is the analogy of \cite[Lemma 27]{BNY21}, which is the key lemma for the inductive construction of game strategy.

\begin{lemma}\label{main-lemma}
    Let $A,\alpha$ be as in Definition \ref{ahlfors}, $\rho_*=\min\{\rho,1\}$, where $\rho$ is as in Proposition \ref{nondeg-global}  for  $B(\bx_0,2r_0)$  in the nondegenerate case, and $\rho=\rho_{B(\bx_0,2r_0)}$ in the affine case, and  $C,\gamma$ be as in Theorem \ref{analytic-nde}, and let
    \begin{equation*}
        C'=CA^2\max\{2^\alpha(\rho_*)^{-\gamma},3^\alpha2^{\gamma}\}.
    \end{equation*} Also, $s$ is as defined in \eqref{constant-s}.

    Then for any quintuple of nonnegative integers $(h,k,l,m,N)$, if
    \begin{equation}\label{removed-1}
        d_{l+m}a_kH_0(\tilde{\bx}_N)\Z^{n+1}\in K_{\sqrt{n+1}\beta^{\frac{m}{n+1}}r_0^{-1}},
    \end{equation}
    and\begin{equation}\label{removed-2}
        d_{l}a_{k-h}H_0(\tilde{\bx}_N)\Z^{n+1}\in K_{\sqrt{2}n\beta^{\frac{\tau}{n}}}
    \end{equation}
    for some $\tilde{\bx}_N\in B_N$, where\begin{equation}
        \tau:=\min\{k-l-m-N-\frac{h}{1+w_1},\frac{hw_n}{1+w_1}\}\geq0,k<2N,l<N,\label{quin-restriction}
    \end{equation}
then for any $\epsilon >0,$\begin{equation}\label{pure-measure-bound}
    \mu(A_{\epsilon}^{k,l,N})\leq C'\epsilon^\gamma\mu(B_N).
\end{equation}
If $N=0$ , $l\geq1$  and $k\geq sl$, then \eqref{pure-measure-bound} holds without assuming \eqref{removed-1}, \eqref{removed-2} and \eqref{quin-restriction}.

Moreover, if $0<r\leq\beta^Nr_0$ and \begin{equation}
    \epsilon':=2(1+\max\{\beta^{l-k},b^{(1+\tilde{w}_{t+1})k}\}r)\epsilon,
\end{equation} 
where $\tilde{w}_{t+1}=w_{t+1}$ if $0< t<d$ and $\tilde{w}_{t+1}=0$ if $t\geq d$,then\begin{equation}
    \mu(B(A_{\epsilon}^{k,l,N},r))\leq C'\epsilon'^\gamma\mu(B_N),\label{nbhd-measure-bound}
\end{equation}
which is valid for   $N\geq1$ under  \eqref{removed-1}, \eqref{removed-2} and \eqref{quin-restriction} . For  $N=0$ , it holds when $l\geq1$,  $k\geq sl$  and, additionally, $r\leq\beta^{k-l}r_0$ .

 \end{lemma}
 \begin{proof}
     Write   $\bx=\tilde{\bx}_N+\bx'$. For the lower-bound argument we take  $\Vert\bx'\Vert\leq\beta^Nr_0$. Moreover, by \cite[Equation (35)]{BNY21} and the fact that $\tilde{\bx}_N\in B_N=X\cap B(\bx_N,\beta^Nr_0),$ we have\begin{equation}
         B(\bx_N,2\beta^Nr_0)\subset B(\tilde{\bx}_N,3\beta^Nr_0)\subset B(\bx_N,4\beta^Nr_0)\subset B(\bx_0,5r_0)\subset B(\bz).\label{chain-of-balls}
     \end{equation}

     We first deal with the case assuming \eqref{removed-1}, \eqref{removed-2} and \eqref{quin-restriction}. We claim that for any  $1\leq j\leq n$ and any nonzero primitive $\bv\in\bigwedge^j\Z^{n+1}$
      \begin{equation}\label{sup-lower-bound}
         \sup_{\Vert\bx'\Vert\leq\beta^Nr_0}\Vert d_la_kH_0(\bx)\bv\Vert\geq\frac{1}{2}.
     \end{equation}

     Assuming toward contradiction that for some non-zero $\bv=\bv_1\wedge...\wedge\bv_j\in\bigwedge^j(\Z^{n+1})$, $\sup_{\Vert\bx'\Vert\leq\beta^Nr_0}\Vert d_la_kH_0(\bx)\bv\Vert<\frac{1}{2}$, then by \eqref{quin-restriction} and \eqref{transition}, and denoting $$\bv'=d_la_kH_0(\tilde{\bx}_N)\bv,\bv_i'=d_la_kH_0(\tilde{\bx}_N)\bv_i,1\leq i\leq j,$$ we have \begin{equation}
     \label{reduced-sup-bound}
         \sup_{\Vert\bx'\Vert\leq\beta^Nr_0}\Vert u_{(\beta^{l-k}x_1',...,\beta^{l-k}x_t',b^{(1+w_{t+1})k}x_{t+1}',...,b^{(1+w_d)k}x_d',0,...,0)}\bv'\Vert<1.
     \end{equation}

     By applying Minkowski's convex body theorem, we may assume without loss of generality that \begin{equation}\Vert\bv_1'\Vert<\sqrt{j}\leq\sqrt{n}\leq r_0^{-1}\label{v1-bound}.\end{equation}

     If $|v_{1,i}'|<\beta^mr_0^{-1},$ for all $1\leq i\leq t,$ where $\bv_1'=(v_{1,+}',v_{1,1}',...,v_{1,n}'),$ then using \eqref{v1-bound} we get that  \begin{equation}
         \Vert d_{l+m}a_kH_0(\tilde{\bx}_N)\bv_1\Vert=\Vert d_m\bv_1'\Vert<\sqrt{n+1}\beta^{\frac{tm}{n+1}}r_0^{-1}\leq \sqrt{n+1}\beta^{\frac{m}{n+1}}r_0^{-1},
     \end{equation} 
     and since $\bv_1\in\Z^{n+1}$ is non-zero, it contradicts \eqref{removed-1}.

     If there is some $1\leq i_0\leq t$ such that   $|v_{1,i_0}'|\geq\beta^mr_0^{-1}$, then, since  by \eqref{def-of-matrix} we have  $1\leq i_0\leq t\leq d$,  we may use the form $\bx'=(0,0,...,0,x_{i_0}',0,...,0)$, that is, $i_0-$th entry is the only non-zero entry. In this case, we define
       \begin{equation}
         \tilde{\bv}=\sum_{i=1}^j(-1)^{i+1}v_{i,i_0}'\bigwedge_{i'\neq i}\bv_{i'}'
         \label{v-tilde}.
     \end{equation} 

     Then by \cite[Proposition 24]{BNY21}, we have\begin{equation}
         u_{(\beta^{l-k}x_1',...,\beta^{l-k}x_t',b^{(1+w_{t+1})k}x_{t+1}',...,b^{(1+w_d)k}x_d',0,...,0)}\bv'=\bv'+\beta^{l-k}x_{i_0}'\be_{+}\wedge\tilde{\bv}
     \end{equation}
     for all $\bx'=(0,0,...,0,x_{i_0}',0,...,0).$ Then by \eqref{reduced-sup-bound}, we have $\Vert \bv'\pm\beta^{l+N-k}r_0\be_{+}\wedge\tilde{\bv}\Vert<1$ when $x_{i_0}'=\pm\beta^Nr_0.$ Therefore, by  the triangle  inequality, we have  \begin{equation}
         \Vert 2\beta^{l+N-k}r_0\be_{+}\wedge\tilde{\bv}\Vert\leq \Vert \bv'+\beta^{l+N-k}r_0\be_{+}\wedge\tilde{\bv}\Vert+\Vert \bv'-\beta^{l+N-k}r_0\be_{+}\wedge\tilde{\bv}\Vert<2.
     \end{equation} 

     Simplifying it, we get
     \begin{equation}\label{norm-bound}
         \Vert \be_+\wedge \tilde{\bv}\Vert<\beta^{k-l-N}r_0^{-1}.
     \end{equation}
     In particular, since $\be_+\wedge\tilde{\bv}_+=0$, we may focus on $\be_+\wedge\tilde{\bv}_-.$ Moreover, notice that  $\Vert \be_+\wedge\tilde{\bv}_-\Vert=\Vert\tilde{\bv}_-\Vert,$  \eqref{norm-bound} can be rewritten as
     \begin{equation}\label{v--bound}
         \Vert\tilde{\bv}_-\Vert<\beta^{k-l-N}r_0^{-1}.
     \end{equation}
     On the other hand, noticing $\bv_1'\wedge\tilde{\bv}=v_{1,i_0}'\bv'$, by \eqref{v1-bound} we can again transform \eqref{v--bound} into  \begin{equation}
         \Vert\bv_-'\Vert\leq|v_{1,i_0}'|^{-1}\Vert\bv_{1,-}'\Vert\Vert\tilde{\bv}_{-}\Vert<\sqrt{n}\beta^{k-l-m-N}.
     \end{equation} 
     Applying \cite[Proposition 24]{BNY21} to the equation above we have\begin{equation}\label{v'-bound-1}
         \Vert a_{-h}\bv_-'\Vert<\sqrt{n}\beta^{k-l-m-N-\frac{h}{1+w_1}}.
     \end{equation}
     Then by noticing $\Vert \bv_+'\Vert\leq\Vert\bv'\Vert<1$ from \eqref{reduced-sup-bound} and again applying \cite[Proposition 24]{BNY21}, we get\begin{equation}\label{v'-bound-2}
         \Vert a_{-h}\bv'_+\Vert<b^{-w_nh}.
     \end{equation}
     Now by combining \eqref{quin-restriction}, \eqref{v'-bound-1} and \eqref{v'-bound-2}, and using  orthogonality , we obtain that  \begin{equation}
         \Vert d_{l}a_{k-h}H_0(\tilde{\bx}_N)\bv\Vert=\Vert a_{-h}\bv'\Vert<\sqrt{2n}\max\{\beta^{k-l-m-N-\frac{h}{1+w_1}},b^{-w_nh}\}=\sqrt{2n}\beta^\tau.
     \end{equation} 
     Then applying Minkowski's convex body theorem, there exists a non-zero vector $\bv_0\in d_{l}a_{k-h}H_0(\tilde{\bx}_N)\Z^{n+1}$ such that\begin{equation}
         \Vert\bv_0\Vert<\sqrt{j}(\sqrt{2n}\beta^\tau)^{\frac{1}{j}}\leq\sqrt{2}n\beta^{\frac{\tau}{n}},
     \end{equation}
     violating \eqref{removed-2}.

     Now, \eqref{sup-lower-bound} holds, and we can apply Theorem \ref{analytic-nde} with $\rho=\frac{1}{2}$ and $B(\tilde{\bx}_N,3\beta^Nr_0)$, to obtain that\begin{equation}\label{measure-bound-1}
         \mu(A_{\epsilon}^{k,l}(B(\tilde{\bx}_N,3\beta^Nr_0)))\leq C2^\gamma\epsilon^\gamma\mu(B(\tilde{\bx}_N,3\beta^Nr_0)).
     \end{equation}
     Then by Definition \ref{ahlfors} we get\begin{equation}\label{measure-bound-2}
         \mu(B(\tilde{\bx}_N,3\beta^Nr_0))\leq A(3\beta^Nr_0)^{\alpha}\leq 3^{\alpha}A^2\mu(B_N).
     \end{equation}
     Hence by \eqref{chain-of-balls}, \eqref{measure-bound-1} and \eqref{measure-bound-2}, we have\begin{equation}\label{main-measure-bound}
         \mu(A_{\epsilon}^{k,l}(B(\bx_N,2\beta^Nr_0)))\leq\mu(A_{\epsilon}^{k,l}(B(\tilde{\bx}_N,3\beta^Nr_0)))\leq C'\epsilon^{\gamma}\mu(B_N).
     \end{equation}
     Since trivially $A_{\epsilon}^{k,l,N}\subset A_{\epsilon}^{k,l}(B(\bx_N,2\beta^Nr_0))$, \eqref{main-measure-bound} implies \eqref{pure-measure-bound} immediately. 

     Secondly, assuming $N=0$ and $k\geq sl\geq \frac{1+w_1}{w_1}l$, we notice that by Definition \ref{ahlfors}\begin{equation}
         \mu(B(\bx_N,2\beta^N r_0))\leq A^22^\alpha\mu(B_N),\label{observation-1}
     \end{equation}and $d_la_k=\text{diag}(b_+,b_1,...,b_n)$ satisfies $k\geq sl$, so we can apply Proposition \ref{nondeg-global}, or respectively Proposition \ref{affine-global}  . For the latter, put $J=\diag(-1,1,\ldots,1)$. Then $H_0(\bx)=Ju_1^\star(\bx)J$, so the same lower bound holds for $H_0$. We then apply Theorem \ref{analytic-nde} with $\rho_*$  to obtain \eqref{main-measure-bound} from \eqref{observation-1}, which immediately implies \eqref{pure-measure-bound}.

     To see \eqref{nbhd-measure-bound}  in  both cases,  let  $\bx\in A_{\epsilon}^{k,l,N}$ . Then there exists a nonzero  $\bv\in \Z^{n+1}$ such that $\Vert d_la_kH_0(\bx)\bv\Vert<\epsilon.$ For some $\by=\bx+\bx'$ with $\Vert\bx'\Vert<r$, we can obtain from \cite[Proposition 24]{BNY21} that \begin{equation}\begin{aligned}
         & u_{(\beta^{l-k}x_1',...,\beta^{l-k}x_t',b^{(1+w_{t+1})k}x_{t+1}',...,b^{(1+w_d)k}x_d',0,...,0)}d_la_kH_0(\bx)\bv\\
         &=d_la_kH_0(\bx)\bv+\langle\tilde{\bv},d_la_kH_0(\bx)\bv\rangle\be_+,
     \end{aligned}\end{equation}
     where $\tilde{\bv}=(0,\beta^{l-k}x_1',...,\beta^{l-k}x_t',b^{(1+w_{t+1})k}x_{t+1}',...,b^{(1+w_d)k}x_d',0,...,0)$. Using 
Lemma \ref{transition} 
(its alternative case when $N=0$), the triangle  inequality, and   Cauchy--Schwarz , we have \begin{equation}\begin{split}
         \Vert d_la_kH_0(\by)\bv\Vert&\leq 2\Vert u_{(\beta^{l-k}x_1',...,\beta^{l-k}x_t',b^{(1+w_{t+1})k}x_{t+1}',...,b^{(1+w_d)k}x_d',0,...,0)}d_la_kH_0(\bx)\bv\Vert\\&\leq2\Vert d_la_kH_0(\bx)\bv\Vert+2|\langle\tilde{\bv},d_la_kH_0(\bx)\bv\rangle|\leq 2(1+\Vert\tilde{\bv}\Vert)\Vert d_la_kH_0(\bx)\bv\Vert.
         \end{split}
     \end{equation}
     Note that $\Vert\tilde{\bv}\Vert\leq\max\{\beta^{l-k},b^{(1+\tilde{w}_{t+1})k}\}r$, since if  $t=d$, then  $$\tilde{\bv}=(0,\beta^{l-k}x_1',...,\beta^{l-k}x_d',0,...,0).$$ Therefore,  \begin{equation}
         \Vert d_la_kH_0(\by)\bv\Vert<2(1+\max\{\beta^{l-k},b^{(1+\tilde{w}_{t+1})k}\}r)\epsilon=\epsilon'.
     \end{equation} 
     Moreover, since $\bx\in A_{\epsilon}^{k,l,N}\subset B_N$ and $r\leq\beta^Nr_0,$ we have $\by\in B(\bx_N,2\beta^Nr_0)$, so the  inequality above gives  $B(A_{\epsilon}^{k,l,N},r)\subset A_{\epsilon'}^{k,l}(B(\bx_N,2\beta^Nr_0))$, then \eqref{main-measure-bound} with $\epsilon$ replaced by $\epsilon'$ implies \eqref{nbhd-measure-bound}.
     \end{proof}

\section{Winning strategy}


This section aims to provide a winning strategy using the measure bound in the previous section.

  Put $\tilde{w}_{t+1}=w_{t+1}$ if $t<d$, and $\tilde{w}_{t+1}=0$ if $t=d$. Denote\begin{equation}\label{constant-s}
    s:=\max\left\{5,\left\lceil\frac{1+w_1}{w_1-\tilde{w}_{t+1}}\right\rceil,\left\lceil\frac{2(1+w_1)+1}{w_1}\right\rceil,s_0\right\},
\end{equation} 
where $s_0$ is   a sufficiently large integer  so that Proposition \ref{affine-global} holds; in the nondegenerate case take $s_0=1$. Let
 \begin{equation}
    \eta:=\min\left\{\frac{1}{4(n+1)},\frac{w_ns}{2n(1+w_1)},\frac{w_1s-2(1+w_1)}{2n(1+w_1)},\frac{\alpha}{\gamma}\right\},\label{constant-eta}
\end{equation}
\begin{equation}
    \label{constant-alpha-prime}
    \alpha':=\alpha-\frac{\gamma\eta}{4s}.
\end{equation}

To begin with, denote  \begin{equation}
    A_{1,i}:=\bigcup_{N'\geq0,l\geq\frac{N'}{s},N'+(s-1)l=i+1}\left\{\bx\in B_0:d_la_{N'+sl}z_{D\varphi(\bx)}u_{\phi(\bx)}\Z^{n+1}\notin K_{\beta^{\eta l}}\right\}.
\end{equation} 
Next, for $N\geq1,$ denote
 \begin{equation}
    A_{N+1,i}:=\left\{\bx\in B_N:d_la_{N+sl}z_{D\varphi(\bx)}u_{\phi(\bx)}\Z^{n+1}\notin K_{\beta^{\eta l}}\right\}
\end{equation} 
if  $i=(s-1)l-1$  for some integer $0<l<\frac{N}{s}$, and we define $A_{N+1,i}:=\emptyset$ otherwise.

Now define
\begin{equation}
    \tilde{A}_{N+1,i}:=A_{N+1,i}\setminus\bigcup_{0\leq N'<N,i'\geq0}A_{N'+1,i'}.
\end{equation}

 Let $\mathcal A_{N+1,i}$ be the cover of $\tilde{A}_{N+1,i}$ by balls in $X$ of radius $\beta^{N+i+1}r_0$ supplied by the covering lemma. Then the outcome of the game $\bx$ satisfies \begin{equation}
    d_la_{N+sl}z_{D\varphi(\bx)}u_{\phi(\bx)}\Z^{n+1}\in K_{\beta^{\eta l}}
\end{equation} 
for all $N\geq0,l\geq1$. Using the case $l=1,$ and \eqref{conjugate-az}, we have  \begin{equation}
    a_Nu_{\phi(\bx)}\Z^{n+1}=z_{O(1)}d_{-1}a_{-s}(d_1a_{N+s}z_{D\varphi(\bx)}u_{\phi(\bx)}\Z^{n+1})\in z_{O(1)}d_{-1}a_{-s}K_{\beta^\eta},
\end{equation}  
where the  right-hand side  is a bounded set that is independent of $N.$ Thus by Dani's correspondence  (in its dual form, see \mbox{
\cite[Equation (25)]{BNY21}}\hskip0pt
), $\bx\in\phi^{-1}(\mathbf{Bad}(\bw))$ .

Now it suffices to show that this strategy is legal. In particular, by \cite[Lemma 22]{BNY21}, it suffices to give an upper bound for  $\mu(B(\tilde{A}_{N+1,i},\frac{1}{3}\beta^{N+i+1}r_0)).$

In the case $N=0$, we have  \begin{equation}
    B(\tilde{A}_{1,i},\frac{1}{3}\beta^{N+i+1}r_0)=\bigcup_{N'\geq0,l\geq\frac{N'}{s},N'+(s-1)l=i+1}B(A_{\beta^{\eta l}}^{N'+sl,l,0},\frac{1}{3}\beta^{N+i+1}r_0).
\end{equation} 

Fix $i\geq0$. For every $N',l$ such that $N'\geq0,l\geq\frac{N'}{s},N'+(s-1)l=i+1$, we can apply Lemma \ref{main-lemma} with quintuple  $(0,N'+sl,l,0,0),\epsilon=\beta^{\eta l},r=\frac{1}{3}\beta^{N+i+1}r_0$ to get $\mu(B(A_{\beta^{\eta l}}^{N'+sl,l,0},\frac{1}{3}\beta^{N+i+1}r_0))\leq C'\epsilon'^\gamma\mu(B_0)\leq C'4^\gamma\beta^{\gamma\eta l},$  where the last inequality follows from the fact that $\mu(B_0)\leq1,$ and  $$\epsilon'=2(1+\max\{\beta^{-(N'+(s-1)l)},b^{(1+\tilde{w}_{t+1})(N'+sl)}\}\frac{1}{3}\beta^{N+i+1}r_0)\beta^{\eta l}\leq4\beta^{\eta l}$$  by \eqref{constant-s}.

      \begin{equation}\begin{split}\label{bound-1}
        \mu(B(\tilde{A}_{1,i},\frac{1}{3}\beta^{N+i+1}r_0))&\leq\sum_{N'\geq0,l\geq\frac{N'}{s},N'+(s-1)l=i+1}C'4^\gamma\beta^{\gamma\eta l}\\&\leq\sum_{l\geq\frac{i+1}{2s-1}}C'4^\gamma\beta^{\gamma\eta l}\leq\frac{C'4^\gamma}{1-\beta^{\gamma\eta}}\left(\beta^{\frac{\gamma\eta}{2s-1}}\right)^{i+1}\\&\leq C'4^{\gamma+1}\beta^{\frac{\gamma\eta(i+1)}{2s-1}}\beta^{\alpha N}
        \end{split}
    \end{equation} 
since $N=0.$

In the case $N\geq1,$ without loss of generality, $\tilde{A}_{N+1,i}\neq\emptyset.$ Then by construction, picking a point $\tilde{\bx}_N\in \tilde{A}_{N+1,i}\subset B_N,$ we have $\tilde{\bx}_N\notin\bigcup_{0\leq N'<N,i'\geq0}A_{N'+1,i'}.$ In particular,  \begin{align*}
     &\zblue{\tilde{\bx}_N\not\in} A_{1,N+(s-2)\ell-1}\cup A_{1,N-\ell-1}  & \text{ if } &  & \frac{N}{2s} \le  \ell < \frac{N}{s}\,,& \\
     &\zblue{\tilde{\bx}_N\not\in}  A_{1,N+(s-2)\ell-1} \cup A_{N+1-s\ell,(s-1)\ell-1}  & \text{ if } & & \frac{N}{3s}  \le  \ell < \frac{N}{2s}\,, &\\
     &\zblue{\tilde{\bx}_N\not\in} A_{N+1-s\ell,2(s-1)\ell-1}\cup A_{N+1-s\ell,(s-1)\ell-1}& \text{ if } & & \phantom{\frac{N}{3s} < {}} \ell < \frac{N}{3s}\,.&
 \end{align*} 
 A routine inspection of each of the sets above gives that
 \begin{equation}\label{vb3001}
d_{2\ell}a_{N+s\ell}H_0(\tilde{\bx}_N)\bbZ^{n+1} 	 \in K_{\beta^{2\eta\ell}} \subseteq K_{\sqrt{n+1}\beta^{\frac{\ell}{n+1}}r_0^{{-1}}}\,,
\end{equation} 
and  \begin{equation}\label{vb3002}
d_{\ell}a_{N}H_0(\tilde{\bx}_N)\bbZ^{n+1}  		 \in K_{\beta^{\eta\ell}} \subseteq K_{{\sqrt{2}n}\beta^{\frac{\tau}{n}}}\,,
\end{equation} 
where

\[
\zblue{\textstyle
\tau=\min\left\{k-\ell-m-N-\frac{h}{1+w_1},\frac{hw_n}{1+w_1}\right\}=\min\left\{ (s-2)\ell - \frac{s\ell}{1+w_1},\frac{s\ell w_n}{1+w_1}\right\}>0\,.}
\]

The containments on the  right-hand side follow by taking  $\beta$  sufficiently small, as in \mbox{
\cite[Equation (66)]{BNY21}}\hskip0pt
.

The final prerequisite to apply Lemma \ref{main-lemma} is $N+sl<2N$ and $l<N$, which clearly holds by the fact that $l<\frac{N}{s}.$

Therefore, by applying Lemma \ref{main-lemma} with quintuples $(sl,N+sl,l,l,N),\epsilon=\beta^{\eta l},r=\frac{1}{3}\beta^{N+(s-1)l}r_0,$ we have $\epsilon'\leq4\beta^{\eta l}$ and thus,  \begin{equation}\begin{split}\label{bound-2}
        \mu(B(\tilde{A}_{N+1,i},\frac{1}{3}\beta^{N+i+1}r_0))&\leq C'4^{\gamma}\beta^{\gamma\eta l}\mu(B_N)\leq C'4^{\gamma}\beta^{\gamma \eta l}A(\beta^Nr_0)^{\alpha}\\&\leq C'Ar_0^{\alpha}4^{\gamma}\beta^{\frac{\gamma\eta (i+1)}{2(s-1)}}\beta^{\alpha N}. 
        \end{split}
    \end{equation} 
 To finish, set $c=\gamma\eta/(4s), \alpha'=\alpha-c$ and
$$D=\max\{C'4^{\gamma+1},C'Ar_0^\alpha4^\gamma\}.$$
The bounds \eqref{bound-1} and \eqref{bound-2} above give
$$\mu\bigl(B(\tilde A_{N+1,i},\tfrac13\beta^{N+i+1}r_0)\bigr)\leq D\beta^{2c(i+1)}\beta^{\alpha N}.$$
Lemma \ref{covering} therefore gives
$$\#\mathcal A_{N+1,i}\leq A3^\alpha r_0^{-\alpha}D\,\beta^{c(i+1)}\beta^{-(\alpha-c)(i+1)}\leq\beta^{-\alpha'(i+1)},$$
provided $\beta$ is small enough that
$$A3^\alpha r_0^{-\alpha}D\,\beta^{c}\leq1.$$


Thus we reach Cantor winning as in Definition \ref{cantor game}, which finishes the proof.

 \section{Appendix}\label{sec: Appendix}

We recall the following lemma and proposition from \cite{BNY21}.
\begin{definition}\label{def-hyp-diff}
    A non-empty closed set $K\subset\R^d$ is called hyperplane diffuse   if there exist  $\beta>0,r_0>0$ such that  for any $\bx\in K$, $0<r\leq r_0$ and any hyperplane $H\subset\R^d$,  we have 
\[
K\cap(B(\bx,r)\setminus B(H,\beta r))\neq\emptyset.
\]

\end{definition}

\begin{proposition}{\cite[Proposition 14]{BNY21}}\label{construction-of-mu}
    If $K$ is hyperplane diffuse, then there exists  a nonzero  Ahlfors-regular absolutely decaying measure $\mu$ so that $\supp\mu\subset K.$
\end{proposition}
\begin{lemma}{\cite[Lemma 15]{BNY21}}\label{axiom}
    For any $\beta>0,$ there  exist  $\beta'\in(0,\beta)$ and  an  integer $N$ so that for every ball $B=B(\bx,\rho)\subset\R^d,$ there is a collection of at most $N$ hyperplanes $\mathcal H_B$ so that for every hyperplane $H'$ there is $H\in\mathcal H_B$ for which 
\[
B(\bx,\rho)\cap B(H',\beta'\rho)\subset B(H,\beta\rho).
\]

\end{lemma}

Now we imitate the proof of  the  second part of \cite[Proposition 13]{BNY21} to finish this section.

\subsection{Proof of  Lemma \ref{equivalence1}} Suppose such $S$ is not HAW in $U.$ Then by \cite[Theorem 1.6]{FLS2014}, Bob has a winning strategy, which will be fixed for the rest of the proof. Moreover, all balls in the proof are  closed  balls in $\R^d$ .  Let $\beta$ and $B_0\cap U$ be chosen on the first move of Bob according to his winning strategy. Define $\mathcal{B}_0 = \{B_0\cap U\}$ and continue by induction to construct collections of closed balls $\mathcal{B}_n$ as follows. Given $\mathcal{B}_n$, for every $B\cap U \in \mathcal{B}_n$ let $\mathcal{H}_B$ be the collection of hyperplanes arising from Lemma \ref{axiom}. Define $\mathcal{B}_{n+1}(B)$ to be the collection of all of Bob's responses according to the winning strategy while considering the hyperplanes in $\mathcal{H}_B$ as possible moves of Alice. Note that $\mathcal{B}_{n+1}(B)$ is always nonempty. Define
\begin{equation}\label{def-K}
    K = \big(\bigcap_{n \geq 0} \bigcup_{B\cap U \in \mathcal{B}_n} B\big)\cap U.
\end{equation}
By Lemma \ref{axiom}, every $\mathbf{x} \in K$ is an outcome of the hyperplane absolute game played according to Bob's winning strategy. Therefore, $\mathbf{x} \notin S$. Since $\mathbf{x}$ is an arbitrary point of $K$, we have that $K \cap S = \varnothing$. Moreover,  by  definition, $K\subset U.$

Now by Proposition \ref{construction-of-mu}, it  remains  to verify that $K$ is hyperplane diffuse. Indeed, we will show that it is $\frac{\beta' \beta}{2}$ -hyperplane  diffuse for $\beta'$ as in Lemma \ref{axiom}. Assume $\mathbf{x} \in K$, $0 < r \leq r_0$ and $H' \subset \mathbb{R}^d$ is a hyperplane, where $r_0$ is the radius of $B_0$. Let $n$ be the unique positive integer such that
\begin{equation}\label{scale-of-r}
    2\beta^n r_0 \leq r < 2\beta^{n-1} r_0,
\end{equation}
which clearly exists since $0 < \beta < 1$. Since $\mathbf{x} \in K$, by \eqref{def-K}, there exists a ball $B = B(\mathbf{x}_0, \beta^n r_0),$ so that $B\cap U \in \mathcal{B}_n$ and $\mathbf{x} \in B$. The left hand side of \eqref{scale-of-r} implies that $B \subseteq B(\mathbf{x}, r)$. By Lemma \ref{axiom} applied with $\rho = \beta^n r_0$, there exists $H \in \mathcal{H}_B$ such that
\[
B \cap B(H', \beta' \beta^n r_0) \subset B(H, \beta^{n+1} r_0).
\]
The right hand side of \eqref{scale-of-r} implies that $\frac{\beta \beta'}{2} r < \beta' \beta^n r_0$ and hence
\[
B \cap B\left(H', \frac{\beta \beta'}{2} r\right) \subset B(H, \beta^{n+1} r_0).
\]
By the definition of $\mathcal{B}_{n+1}(B)$, there exists $B'\cap U \in \mathcal{B}_{n+1}(B)$ such that $B'\cap U \cap B(H, \beta^{n+1} r_0) = \varnothing$. Since the collections $\mathcal{B}_{n+1}(B)$ are always nonempty, by \eqref{def-K}, we have that $K \cap B'=K\cap B'\cap U \neq \varnothing$. Since $\varnothing \neq K \cap B' \subset K \cap B \subset K \cap B(\mathbf{x}, r)$ , we have $K \cap B(\mathbf{x}, r) \not\subset B(H', \frac{\beta \beta'}{2} r)$. Hence, $K \cap (B(\mathbf{x}, r) \setminus B(H', \frac{\beta \beta'}{2} r)) \neq \varnothing$. This verifies Definition \ref{def-hyp-diff} for the set $K$ and thus completes the proof.
\subsection*{Acknowledgment:}
We are grateful to David Simmons for sharing his note comparing higher rank Diophantine exponents with rank $1$. 
We did not use his method of proof, as it was not clear to us how to extend that idea to a weighted setup. LS thanks Zuo Lin, Amir Mohammadi and Lei Yang for many helpful discussions. SD thanks Victor Beresnevich for his constant support and also for their joint work \cite{BD_upcoming} that was also useful in this project. SD thanks the EPSRC grant EP/Y016769/1 and the ISI startup grant.

\subsection*{AI declaration}
The work was completed without use of any AI, which started in 2025 and was completed in the early summer of 2026. After finishing, the authors revised the paper using ChatGPT to correct minor inaccuracies. 

\bibliographystyle{abbrv}
\bibliography{inhomo_curve}

\end{document}